\documentclass[12pt]{amsart}
\usepackage[margin=45pt]{geometry}

\usepackage[foot]{amsaddr}
\usepackage{hyperref}\hypersetup{
    colorlinks=true,
    linkcolor=purple,
    filecolor=blue,      
    urlcolor=cyan,
    }

\usepackage{dsfont}
\usepackage{mathrsfs} 
\usepackage{amssymb, changes}
\usepackage[english]{babel}
\usepackage{tcolorbox}
\usepackage{xcolor}
\usepackage{cleveref}
\usepackage[dvipsnames]{xcolor}
\usepackage{bbm}
\usepackage{tikz}
\usetikzlibrary{arrows.meta}
\usepackage{booktabs}
\usetikzlibrary{positioning}
\usetikzlibrary{arrows.meta,decorations.markings}

\definecolor{lime}{HTML}{A6CE39}
\DeclareRobustCommand{\orcidicon}{%
	\begin{tikzpicture}
	\draw[lime, fill=lime] (0,0) 
	circle [radius=0.16] 
	node[white] {{\sffamily \tiny ID}};
	\draw[white, fill=white] (-0.0625,0.095) 
	circle [radius=0.007];
	\end{tikzpicture}
	\hspace{-2mm}
}
\foreach \x in {A, ..., Z}{%
	\expandafter\xdef\csname orcid\x\endcsname{\noexpand\href{https://orcid.org/\csname orcidauthor\x\endcsname}{\noexpand\orcidicon}}
}

\usepackage{times}
\usepackage{xcolor}
\newtheorem{theorem}{Theorem}[section]
\newtheorem{proposition}[theorem]{Proposition}
\newtheorem{lemma}[theorem]{Lemma}

\newtheorem{definition}{Definition}

\newcommand{\E}{\mathbb{E}}
\renewcommand{\P}{\mathbb{P}}

\def\Tcal{\mathcal T}

\def\Pcal{{\mathcal P}}
\newcommand{\X}{\mathbf{X }}
\newcommand{\C}{\mathbb{C}}

\newcommand{\Z}{\mathbb{Z}}
\newcommand{\N}{\mathbb{N}}

\newcommand{\rmd}{\,\mathrm{d}}

\def\1{\mathds{1}}
\def\1vec{\mathbf{1}}
\def\0{\mathbf{0}}

\begin{document}
\title{Multi-excited random walks on general trees}

\author[]{Jack Mills\orcidB{}}
\address[J. M.]{School of Mathematics, Monash University, 9 Rainforest Walk, Clayton 3800, Victoria, Australia}
\email{jack.mills1@monash.edu}
\author[]{Tuan-Minh Nguyen\orcidA{}}
\address[T.-M. N.]{School of Mathematics, Monash University, 9 Rainforest Walk, Clayton 3800, Victoria, Australia}
\email{tuanminh.nguyen@monash.edu}

\begin{abstract}
We study multi-excited random walks on infinite locally finite trees. The model is described by placing a stack of ``cookies'' at each vertex, with one cookie consumed on each successive visit. As long as cookies remain, the walk is in an excited regime and makes nearest-neighbour steps according to a prescribed sequence of biases. Once the stack is exhausted, the walk is in a non-excited regime and moves uniformly among the neighbouring vertices. The model is closely connected to foraging ecology, where repeated visits to a location can deplete local resources and thereby influence the attractiveness of that location for future visits. 
We prove that the process exhibits a sharp phase transition between recurrence and transience. The critical threshold is determined by the branching-ruin number of the tree, which is equal to the Hausdorff dimension of the tree's boundary under a suitable metric. Our result substantially extends Zerner's recurrence-transience criterion for multi-excited random walks on $\mathbb Z$ [Probab. Theory Related Fields {\bf 133} (2005)] to general trees.

\end{abstract}

\maketitle
\tableofcontents

\section{Introduction}\label{sec:intro}
Excited random walks are random walks with memory whose transition probabilities depend on the number of previous visits to the current site. This dependence is described by a cookie environment, in which each site initially carries a stack of cookies prescribing the transition probabilities on successive visits. The cookies can be viewed as local rewards that are depleted through repeated visits. While cookies remain at the current site, the walker is in the excited regime and consumes the next cookie before making a step according to its prescribed bias. Once the stack at that site is exhausted, the walker enters the non-excited regime there and chooses each neighbour with equal probability on every subsequent visit. This mechanism has a natural connection with foraging ecology, where visits to a location deplete local resources and may alter subsequent movement \cite{Campos2021,BenichouRedner2014,Bhat2017}. On trees, the behaviour of excited random walks is influenced by both the excitation and the branching structure. On regular trees, multi-excited random walks (with no bias in the non-excited regime) are transient \cite[Theorem~1.2(b)]{BS2009}. On trees with subexponential growth, however, recurrence and transience can both occur, depending on the bias parameters and the geometry of the tree. In this paper, we study how the growth of the tree and the excitations determine the phase transition between recurrence and transience.

\subsection{Description of the model and main result}
Let $\mathcal T=(V,E)$ be a locally finite infinite tree rooted at $\rho$, and let $\vec{E}=\{(u,v) : \{u,v\}\in E\}$ be the set of all induced oriented edges. For two adjacent vertices $u$ and $v$,
we write $u\sim v$. For a vertex $v\neq \rho$, we denote its parent by $v^{-1}$.
For $n\in\mathbb N$, let $[n]=\{1,2,\dots,n\}$.
Fix a natural number $M\ge1$ and a collection of positive real numbers
$\boldsymbol\lambda=(\lambda_v^{(j)})_{v\in V, j\ge1}$. In the setting considered below, we will assume that $\lambda_{v}^{(j)}=1$ for all $j\ge M+1$ and at the root we set
$\lambda_\rho^{(j)}=1$ for all
$j\ge1$. Let $p_{\boldsymbol{\lambda}}: \vec{E}\times \N\to (0,1]$ be the function defined as follows:
\begin{itemize}
    \item For each $u\sim \rho$ and $j\ge 1$,  $p_{\boldsymbol{\lambda}}(\rho,u,j)=1/\deg(\rho)$;  
    \item For $j\ge 1$ and $v\in V\setminus\{\rho\}$, $$p_{\boldsymbol{\lambda}}(v,v^{-1},j)=\frac{\lambda_v^{(j)}}{\lambda_v^{(j)}+\deg(v)-1} \text{ and } p_{\boldsymbol{\lambda}}(v,u,j)=\frac{1}{\lambda_v^{(j)}+\deg(v)-1},$$  for each child $u$ of $v$.
\end{itemize}
We consider a discrete-time nearest-neighbour random process $\X=(X_n)_{n\ge0}$ on $\Tcal$ with $X_0=\rho$,  whose conditional transition law is given by
\begin{equation}\label{P.quenched}
\P_{\boldsymbol\lambda}(X_{n+1}=y\mid X_0,\dots,X_n) =\mathds{1}_{\{y\sim X_n\}}\,p_{\boldsymbol\lambda}(X_n,y,\mathsf{L}_n(X_n)),
\end{equation}
where
$$
\mathsf{L}_n(x):=\sum_{k=0}^n \mathds 1_{\{X_k=x\}}
$$
is the number of visits to $x$ up to time $n$. In other words, on the $j$-th visit to a vertex $v\neq \rho$, the walk performs a
nearest-neighbour step with \textbf{bias} toward the parent $v^{-1}$, where the bias parameter
$\lambda_v^{(j)}$ is the ratio between the transition probability to the parent $v^{-1}$ and
the transition probability to any child $u$, i.e.,
$$
\lambda_v^{(j)} = \frac{p_{\boldsymbol\lambda}(v,v^{-1},j)}{p_{\boldsymbol\lambda}(v,u,j)}.
$$

The model above admits a natural interpretation in terms of a cookie environment on
$\mathcal T$, where each vertex carries a stack of $M$ cookies. During the first $M$
visits to a vertex $v$, the walk is in an excited regime and, on the $j$-th visit,
moves according to the bias parameter $\lambda_v^{(j)}$. After all cookies at $v$
have been consumed, the walk enters the non-excited regime and performs a simple
symmetric step to one of the neighbours of $v$ (equivalently, $\lambda_v^{(j)}=1$ for $j\ge M+1$ and $v\in V$).

This process is referred to in the literature as a \textbf{multi-excited random walk}
(MERW), and we denote it
by ${\rm MERW}(\boldsymbol\lambda)$. We say that the process $\X$ is \begin{itemize}
    \item \textbf{recurrent} if the process visits each vertex in $\Tcal$ infinitely often,
    \item \textbf{transient} if the process visits each vertex in $\Tcal$ finitely often.
\end{itemize}

\begin{figure}[h]
\centering
\begin{tikzpicture}[
  x=1cm,y=1cm,
  tree edge/.style={gray!70,line width=.65pt},
  tree vertex/.style={circle,draw=black,fill=white,
    inner sep=0pt,minimum size=4pt},
  cookie/.style={circle,draw=orange!80!black,fill=orange!30,
    inner sep=0pt,minimum size=4mm},
  jump/.style={-{Latex[length=2mm]},blue!70!black,line width=1pt},
  walker/.style={circle,fill=blue!70!black,draw=white,
    inner sep=0pt,minimum size=5pt}
]
\newcommand{\ERWtree}{%
  \draw[tree edge]
    (0,2.5)--(0,3.15)
    (0,2.5)--(-1.1,1.45) (0,2.5)--(1.1,1.45)
    (-1.1,1.45)--(-1.7,.3) (-1.1,1.45)--(-.5,.3)
    (1.1,1.45)--(.35,.3) (1.1,1.45)--(1.8,.3);
  \foreach \x/\y in {0/2.5,-1.1/1.45,1.1/1.45,
    -1.7/.3,-.5/.3,.35/.3,1.8/.3}
    \node[tree vertex] at (\x,\y) {};
  \node[left] at (-.04,2.5) {$u$};
  \node[right] at (1.13,1.4) {$v$};
}
\begin{scope}[shift={(0,0)}]
  \ERWtree
  \foreach \x/\y in {.33/2.63,-1.42/1.58,1.66/1.07,
    -1.94/.40,-.20/.40,.35/.74,2.10/.40} {
    \node[cookie] at (\x,\y) {};
    \node[cookie] at (\x,{\y+.27}) {};
  }
  \node[above left=-5pt] at (0,2.95) {$\lambda_u^{(1)}$};
  \node[left=2pt] at (-.55,1.98) {$1$};
  \draw[jump] (.04,2.46)--(1.06,1.49)
    node[pos=.50,above right=-2pt] {$1$};
  \node[walker] at (0,2.5) {};
  \node at (.1,-.67) {$u\to v$};
\end{scope}
\begin{scope}[shift={(5.2,0)}]
  \ERWtree
  \foreach \x/\y in {-1.42/1.58,1.66/1.07,
    -1.94/.40,-.20/.40,.35/.74,2.10/.40} {
    \node[cookie] at (\x,\y) {};
    \node[cookie] at (\x,{\y+.27}) {};
  }
  \node[cookie] at (.33,2.63) {};
  \node[above right=2pt] at (.58,.56) {$1$};
  \node[left=2pt] at (1.75,.55) {$1$};
  \draw[jump] (1.08,1.48)--(.04,2.46)
    node[pos=.3,above right=-2pt] {$\lambda_v^{(1)}$};
  \node[walker] at (1.1,1.45) {};
  \node at (.1,-.67) {$v\to u$};
\end{scope}
\begin{scope}[shift={(10.4,0)}]
  \ERWtree
  \foreach \x/\y in {-1.42/1.58,-1.94/.40,-.20/.40,
    .35/.74,2.10/.40} {
    \node[cookie] at (\x,\y) {};
    \node[cookie] at (\x,{\y+.27}) {};
  }
  \node[cookie] at (.33,2.63) {};
  \node[cookie] at (1.3,1.8) {};
  \node[above left=-5pt] at (0,2.95) {$\lambda_u^{(2)}$};
  \node[left=2pt] at (-.55,1.98) {$1$};
  \draw[jump] (.04,2.46)--(1.06,1.49)
    node[pos=.50,above right=-1pt] {$1$};
  \node[walker] at (0,2.5) {};
  \node at (.1,-.67) {$u\to v$};
\end{scope}
\end{tikzpicture}
\caption{Three successive steps $u\to v\to u\to v$ of a walk started at $u$. Each vertex initially has a stack of $2$ cookies. }
\label{fig:erw}
\end{figure}
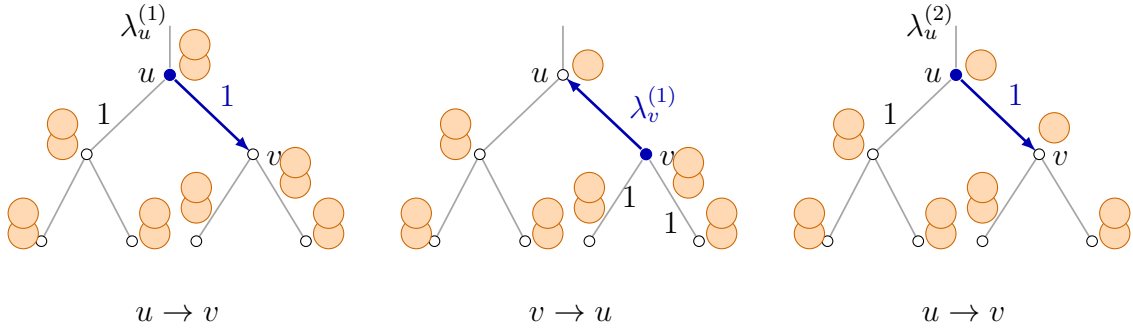

The concept of excited random walks with a single cookie at each vertex on $\Z^d$ was first introduced by Benjamini and Wilson \cite{BW2003} and later extended to multi-excited random walks by Zerner \cite{Z2005}. In recent years, significant attention has been devoted to the study of one-dimensional nearest-neighbor excited random walks, leading to remarkable phase transition results for the asymptotic behavior of the model, including criteria for recurrence/transience \cite{ABO, Z2005} and non-ballisticity/ballisticity \cite{BasS1, BasS2}, as well as the characterization of the limit distribution in specific regimes — see \cite{DK2012, KM2011, KZ2013, P2012, KP2017, KMP2022} for more details. One-dimensional excited random walks with non-nearest-neighbor steps on $\Z$ were studied recently in \cite{DP2017, CHN2021, N2022}. For excited random walks on $\mathbb{Z}^d$, Benjamini and Wilson \cite{BW2003} proved positive speed for $d\ge 4$, while Kozma \cite{Kozma2003, Kozma2005} established ballisticity in dimensions $2$ and $3$. Zerner \cite{Z2006} gave criteria for recurrence and transience in $\mathbb{Z}^d$ and strips. B\'erard and Ram\'irez \cite{BR2007} subsequently proved a ballistic law of large numbers and a central limit theorem for all $d\ge 2$. These results were extended to more general multidimensional excited walks by Menshikov, Popov, Ram\'irez, and Vachkovskaia \cite{Menshikov2012}; see also \cite{Hofstad2010, Pham2015} for further results on the speed, including monotonicity and regularity in higher dimensions. Recent results on shape theorems for random walks excited to the center can be found in \cite{N2026b} and \cite{BNSP2026}. For results related to other variations of excited random walks on $\mathbb{Z}^d$, see \cite{ABK2008, AHR2023}. For a literature review, we refer the reader to \cite{KZ2013}.

Excited random walks on regular trees were first considered by Volkov in \cite{V2003}. In \cite{BS2009}, Basdevant and Singh established a phase transition between transience and recurrence for multi-excited random walks on regular trees in the case where the process has a bias in the non-excited regime. In contrast, for the original case in which there is no bias in the non-excited regime, they showed that the process is always transient on regular trees. This sharp dichotomy naturally suggests that the process should exhibit a recurrence-transience phase transition on trees with sub-exponential growth. Indeed, on trees with polynomial growth, sharp phase transitions were established for both $M$-digging random walks in \cite{CHK2019} and once-excited random walks in \cite{LN2026}. Both models are special cases of multi-excited random walks. The behaviours of these models depend on their ranges, which can be studied using gambler's ruin coupling techniques as in \cite{CKS2020}. Such techniques, however, are not applicable to general multi-excited random walks, whose behaviours depend more strongly on their past trajectories. The present paper extends the results in \cite{BS2009, CHK2019, LN2026} and completes the recurrence-transience phase transition picture for multi-excited random walks on general trees. Our approach differs significantly from those used in \cite{BS2009, CHK2019, CKS2020, LN2026} (see the strategies of proof in Section \ref{sec:sketch} below).

For each edge $e=\{v^{-1},v\}\in E$, let $|e|:=|v|$ denote the distance from $v$ to the root $\rho$, that is, the number of edges in the unique shortest path from $v$ to $\rho$. A \textbf{cutset} in $\mathcal T$ is a minimal set $\pi$ of edges that separates the root from infinity. Equivalently, every infinite self-avoiding path $(v_i)_{i\ge0}$ with $v_0=\rho$ intersects $\pi$ in exactly one edge, i.e., there exists a unique index $i$ such that $\{v_i,v_{i+1}\}\in\pi$. We write $\Pi$ for the collection of all cutsets in $\mathcal T$.
To quantify polynomial growth of trees, we use the \textbf{branching-ruin number}, introduced in \cite{CKS2020}:
\begin{equation}\label{brr2}
\mathrm{br}_r(\mathcal T)
=
\sup\left\{\gamma>0:\inf_{\pi\in\Pi}\sum_{e\in\pi}|e|^{-\gamma}>0\right\}.
\end{equation}
Let $\partial\mathcal T$ denote the \textbf{boundary} of $\mathcal T$, namely the set of all infinite self-avoiding paths starting from the root. For $\xi,\eta\in\partial\mathcal T$, let $e$ denote their last common edge. We define
\begin{equation}\label{def.d}
d(\xi,\eta):=\frac{1}{|e|},
\end{equation}
with the convention that $d(\xi,\eta)=1$ if $\xi$ and $\eta$ have no common edge. It is shown in Section 3.3 of \cite{CKS2020} that $\mathrm{br}_r(\mathcal T)$ coincides with the \textbf{Hausdorff dimension} of $\partial\mathcal T$ with respect to the metric $d$. It is worth noting that the branching-ruin number has also been shown to serve as a critical parameter for several self-interacting random walks on trees with polynomial growth, including once-reinforced random walks \cite{CKS2020}, random walks in random conductances \cite{CHK2019}, and true self-avoiding walks \cite{N2026}.     

In this paper, we establish a phase transition between recurrence and transience for the multi-excited random walk on the tree $\mathcal{T}$ with respect to the branching-ruin number $\text{br}_r(\mathcal{T})$. \vspace{5pt}

\begin{theorem}\label{main.thm}
    Assume that 
    $$\lambda_v^{(j)}=1+\alpha_j \deg(v)$$
    where $\alpha_j\in [0,\infty)$ for $1\le j\le M$ and $\alpha_j=0$ for $j\ge M+1$.
    Let 
    $$D:=\sum_{j=1}^M\frac{\alpha_j}{1+\alpha_j}.$$
    Then, the process ${\rm MERW}(\boldsymbol{\lambda})$ is a.s. transient when ${\rm br}_r(\Tcal)> D+1$ and it is a.s. recurrent when ${\rm br}_r(\Tcal)< D+1$. 
\end{theorem}

Our work generalises the results for once-excited random walks on general trees in \cite{LN2026}, and for multi-excited random walks on $\Z$ in \cite{Z2005}.

Excited random walks have a strong conceptual connection with foraging ecology, where an animal's visits to a location may deplete local resources or alter the value of subsequent visits \cite{Campos2021,BenichouRedner2014,Bhat2017}. This connection was mentioned in \cite{Campos2021}, where excited random walks are used to model feeding and escape dynamics in a progressively depleted environment and to investigate how local resource consumption shapes search strategies. More broadly, models of animal movement show that spatial memory and information about previously visited resource patches can strongly influence foraging efficiency, revisit patterns, and the emergence of home ranges \cite{Bracis2015,VanMoorter2009,Fronhofer2013,GetzSaltz2008,BartumeusLevin2008}. Thus, the visit-dependent transition mechanism of multi-excited random walks provides a natural and mathematically tractable framework for representing feedback between a forager's movement history and the changing attractiveness of its environment.

\textbf{AI disclosure.} All ideas, concepts, and mathematical proofs presented in this paper were developed and written entirely by the authors. Microsoft 365 Copilot (GPT-5.6) was used only for proofreading and for verifying computational results.

\subsection{Structure of the paper and strategies of the proof}\label{sec:sketch}

The rest of the paper is devoted to proving Theorem~\ref{main.thm}. In Section~\ref{sec:strong-const}, we describe a strong construction of $\X=(X_n)_{n\ge0}$ using independent exponential random variables. For each vertex $v\ne\rho$, let $\mathcal P_v$ be the unique shortest path connecting $\rho$ to $v$. The restriction of $\X$ to $\mathcal P_v$ records its positions on this path, ignoring movements outside $\mathcal P_v$ and removing consecutive repetitions of the same vertex. If $\X$ is transient, $\mathcal P_v$ is visited only finitely many times, and this restriction process contains only finitely many recorded positions. Using the same exponential random variables, we construct the extension process $\X^{(v)}=(X_k^{(v)})_{k\ge0}$ on $\mathcal P_v$, with $X_k^{(v)}$ defined for every $k\ge0$, so that the extension agrees with the restriction for as long as the restriction is defined. In particular, on the event that $\X$ visits $\mathcal P_v$ infinitely many times, the restriction agrees with $\X^{(v)}$ at every step (see In Proposition~\ref{prop.res}).

In Section~\ref{sec:backward}, we fix a path $\mathcal{P}_{v_n}=(v_k)_{0\le k\le n}$, with $v_0=\rho$, and consider the inhomogeneous Markov chain $Y^{(n)}=(Y_k^{(n)})_{0\le k\le n-1}$ describing the numbers of backward steps made by $\X^{(v_n)}$ before its first visit to $v_n$. This representation expresses the ruin probability, namely the probability that $\X^{(v_n)}$ reaches $v_n$ before returning to $\rho$, as the probability that $Y^{(n)}$, started at $0$, is at $0$ after $n-1$ transitions. To study this probability, we compare $Y^{(n)}$ with a homogeneous Markov chain $\widetilde Y$, whose transition law depends only on $M$ and the bias parameters $(\alpha_k)_{1\le k\le M}$. We compute the probability generating functions of $\widetilde Y_n$ for each initial state and obtain their asymptotics in Lemmas~\ref{lem.G} and~\ref{lem:G.asym}. We then use generating function identities and stochastic dominations techniques to prove the ruin probability asymptotics in Proposition~\ref{thm:ruin_prob}. In Proposition~\ref{lem:genruin}, we also bound the generalised ruin probability that $\X^{(v_n)}$ reaches $v_n$ before its $j$-th return to $\rho$, uniformly over $j\ge1$ and finite paths. In Section~\ref{sec:forward}, we consider the inhomogeneous Markov chain $Z^{(n)}$ describing the numbers of forward steps along the path $\mathcal{P}_{v_n}$ made by $\X^{(v_n)}$ before its first return to $\rho$. We prove a moment bound for $Z^{(n)}$ in Proposition~\ref{prop.moment-drift}.

In Section~\ref{sec:perc}, we recall general results on quasi-independent percolations and define the ruin percolation. An edge $\{v^{-1},v\}$ is open if $\X^{(v)}$ reaches $v$ before returning to $\rho$, and closed otherwise. The open cluster containing $\rho$ is the connected component of $\rho$ after all closed edges have been removed. By Lemma~\ref{per2.lem}, $\X$ is recurrent if cluster contains only finitely many vertices almost surely, and transient if this cluster contains infinitely many vertices. We prove that the ruin percolation is quasi-independent in the sense of Lyons \cite{Lyons1989}. We first use Proposition~\ref{lem:genruin} to estimate the conditional probability that two edges are open given their nearest common ancestor is open, as stated in Lemma~\ref{lem:branch_corr}. We then combine this estimate with Propositions~\ref{prop.moment-drift} and~\ref{thm:ruin_prob} to prove Proposition~\ref{prop:per}. We also prove a $0-1$ law for recurrence and transience. Finally, we apply the percolation criteria in Proposition~\ref{lyons1989} to prove Theorem~\ref{main.thm}. For recurrence, we use cutsets to prove that the open cluster containing $\rho$ has only finitely many vertices almost surely. For transience, we construct a non-zero flow of finite energy for the adapted conductances of the ruin percolation, and prove that this cluster contains infinitely many vertices with positive probability.

\begin{figure}[h]
\centering
\begin{tikzpicture}[
    scale=0.7,
    transform shape,
    result/.style={
        draw,
        rounded corners=2pt,
        align=center,
        text width=3cm,
        minimum height=1cm,
        inner sep=4pt,
        font=\normalsize
    },
    dependency/.style={
        -{Stealth[length=2mm]},
        line width=0.6pt,
        rounded corners=2pt
    }
]
\node[result] (G) at (0,3.6)
    {Lemma~\ref{lem.G}};
\node[result] (A) at (3.6,3.6)
    {Lemma~\ref{lem:G.asym}};
\node[result] (L) at (7.2,3.6)
    {Lemma~\ref{prop:first-hit}};
\node[result] (H) at (11,3.6)
    {Lemma~\ref{lem:ruinhomo}};
\node[result] (R) at (14.6,3.6)
    {Proposition~\ref{thm:ruin_prob}\\Ruin probability asymtotics};
\node[result] (Z) at (18.4,4.6)
    {Proposition~\ref{prop.moment-drift}\\Moment bounds for $Z^{(n)}$};
\node[result] (LY) at (22,4.6)
    {Proposition~\ref{lyons1989}\\Percolation criteria};
\node[result] (M) at (14.6,5.9)
    {Lemma~\ref{lem:monotone}\\Monotonicity of ruin probabilities};
\node[result] (E) at (11,1.8)
    {Lemma~\ref{lem:escape}};
\node[result] (C) at (14.6,1.8)
    {Lemma~\ref{lem:ruincomp}};
\node[result] (Q) at (18.4,1.8)
    {Proposition~\ref{prop:per}\\Quasi-independent percolation};
\node[result,fill=green!10,line width=1pt] (T) at (22,1.8)
    {Theorem~\ref{main.thm}};
\node[result] (F) at (0,-0.4)
    {Lemma~\ref{lem:F_pgf}};
\node[result] (S) at (3.6,-0.4)
    {Lemma~\ref{lem:small_state2}};
\node[result] (J) at (11,-0.4)
    {Proposition~\ref{lem:genruin}\\Generalised ruin probability};
\node[result] (B) at (14.6,-0.4)
    {Lemma~\ref{lem:branch_corr}};
\node[result] (P) at (14.6,-2.5)
    {Proposition~\ref{prop.res}\\Restriction principle};
\node[result] (O) at (18.4,-2.5)
    {Lemma~\ref{lem:zero.one}\\$0$--$1$ law};
\node[result] (W) at (22,-2.5)
    {Lemma~\ref{per2.lem}\\Connection to percolation};
\draw[dependency] (F.east) -- (S.west);
\draw[dependency] (G.east) -- (A.west);
\draw[dependency] (A.east) -- (L.west);
\draw[dependency] (L.east) -- (H.west);
\draw[dependency] (S.south)
    -- (3.6,-1.4)
    -- (-2.4,-1.4)
    -- (-2.4,4.8)
    -- (11,4.8)
    -- (H.north);
\draw[dependency] (S.east)
    --  (J.west);
\draw[dependency] (A.south)
    -- (3.6,2.7)
    -- (7.2,2.7)
    |- ([yshift=3mm]J.west);
\draw[dependency] (E.north) -- (H.south);
\draw[dependency] (E.south) -- (J.north);
\draw[dependency] (H.east) -- (R.west);
\draw[dependency] (M.south) -- (R.north);
\draw[dependency] (C.north) -- (R.south);
\draw[dependency] (J.east) -- (B.west);
\draw[dependency] (R.east)
    -- (16.4,3.6)
    |- ([yshift=3mm]Q.west);
\draw[dependency] (B.east)
    -- (16.4,-0.4)
    |- ([yshift=-3mm]Q.west);
\draw[dependency] (Z.south) -- (Q.north);
\draw[dependency] (P.east) -- (O.west);
\draw[dependency] (O.east) -- (W.west);
\draw[dependency] (P.north) -- (B.south);
\draw[dependency] (Q.east) -- (T.west);
\draw[dependency] (LY.south) -- (T.north);
\draw[dependency] (W.north) -- (T.south);
\end{tikzpicture}
\caption{Dependencies of the results used in the proof of Theorem~\ref{main.thm}.}
\end{figure}
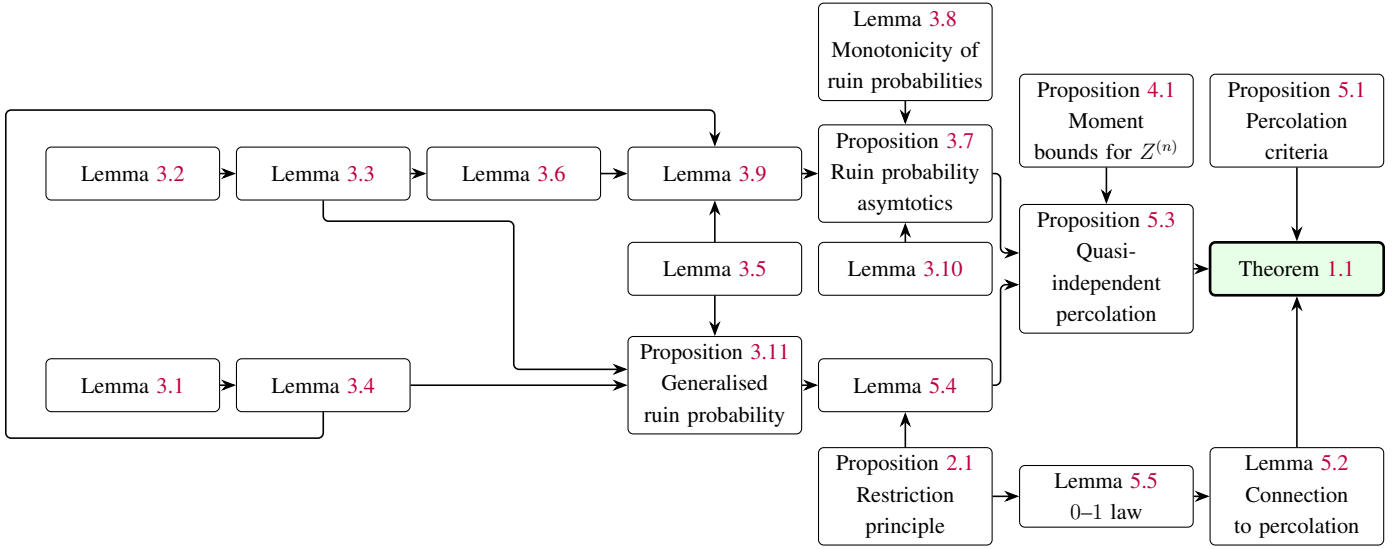

We extend the percolation framework developed in \cite{CKS2020,LN2026} to multi-excited random walks through a new analysis of inhomogeneous Markov chains of backward steps. For the once-reinforced random walks and once-excited random walks studied in \cite{CKS2020,LN2026}, the transition probabilities depend on the set of traversed edges and the set of visited vertices, respectively. In both models, the ruin probabilities on a path can be computed explicitly via the classical gambler's ruin problem for Markovian random walks. For multi-excited random walks, however, the transition probabilities at a vertex depend strongly on the past trajectory, and therefore this approach no longer works. Our approach is built on the representation by Markov structures describing backward steps recently developed by the second author for true self-avoiding walks in \cite{N2026}. In the true self-avoiding walk model, the extension process to a finite path is itself a true self-avoiding walk, and the corresponding Markov chain describing backward steps is homogeneous. In the present model, cookies at a vertex of a path $\mathcal{P}_{v_n}$ may be consumed by steps to neighbours outside that path. Consequently, even when the original cookie environment is fixed, the extension process $\X^{(v_n)}$ defined on $\mathcal{P}_{v_n}$ is a  \textit{random walk in an inhomogeneous random cookie environment}, since cookies that prescribe steps outside $\mathcal{P}_{v_n}$ are omitted and vertices on this path may have different degrees. The transition laws of the Markov chain $Y^{(n)}$ describing backward steps therefore depend on the degrees of the vertices in the original tree, and this Markov chain is generally inhomogeneous. The main difficulty is to analyse this inhomogeneous Markov chain without assuming uniformly bounded vertex degrees. Our key observation is that there exists a homogeneous Markov chain $\widetilde Y$ whose transition law agrees with that of $Y^{(n)}$ whenever $Y^{(n)}$ is at least $M-1$,
independently of the degrees along the path $\mathcal{P}_{v_n}$. We exploit this property through stochastic domination and generating function methods to obtain exact ruin probability asymptotics with polynomial exponent $D+1$ and bounds that are uniform over all finite paths.

\section{Strong construction}\label{sec:strong-const}

In this section we first give a strong construction of the excited random walk $\X$ using a collection $\boldsymbol{\xi}$ of independent rate one exponential random variables; see Section~\ref{sec:rubin}. For each vertex $v\ne\rho$, let $\mathcal P_v$ be the unique self-avoiding path from the root $\rho$ to $v$. We next construct the extension process $\X^{(v)}=(X_k^{(v)})_{k\ge0}$ defined on $\mathcal P_v$ in Section~\ref{sec:ext}. This process is defined at every step such that it agrees with the restriction of $\X$ to $\mathcal P_v$ up to the last exit of $\X$ from $\mathcal P_v$. In particular, on the event that $\X$ visits $\mathcal P_v$ infinitely often, the two processes agree at every step. This construction will be used in Sections~\ref{sec:backward} and~\ref{sec:forward} to define Markov chains associated with the backward and forward steps of $\X^{(v)}$, respectively, and in Section~\ref{sec:perc} to define the ruin percolation.


\subsection{Rubin's construction}\label{sec:rubin}
Let $\boldsymbol{\xi}= \left(\xi(v,u,j) : v \in V, u \sim v,\, j\geq 1 \right)$ be a collection of independent rate one exponential random variables. We use
$\boldsymbol{\xi}$ to give a strong construction of the process $\mathbf X=(X_n)_{n\ge0}$ path wise. Recall that $\mathsf{L}_n(v) = \sum_{k=0}^n\mathds{1}_{\{X_k =  v\}}$ is the number of visits to vertex $v$ by time $n$.  Let 
\[\mathsf{C}_n(v,u) = \sum_{k=1}^n{\mathds{1}}_{\{X_{k-1} = v, X_k = u\}}\] 
be the number of crossings from $v$ to $u$ up to time $n$. For $v\neq \rho$, define the visit-dependent rates
$$
r_j(v,u) := \mathds 1_{\{u=v^{-1}\}}\lambda_v^{(j)} + \mathds 1_{\{u^{-1}=v\}} \quad\text{for } u\sim v,\, j\ge1.
$$
At the root we set
$$
r_j(\rho,u):=1\quad\text{for } u\sim \rho,\, j\ge1,
$$
which corresponds to the uniform transition rule from $\rho$.

We now define the walk recursively. Set $X_0 = \rho$. On the event $\{X_n = v, \mathsf{L}_n(v) = j\}$, with $1\leq j \le M$, the next position is given by
$$X_{n+1} =\arg\min_{u\sim v}\frac{\xi(v,u,j)}{r_j(v,u)}.$$
For $v\neq\rho$ and $1\le j\leq M$, we now verify that 
$$\P(X_{n+1} = v^{-1}\mid X_n =v, \mathsf{L}_n(v) = j) = \frac{\lambda_v^{(j)}}{\lambda_v^{(j)}+\deg(v)-1}.$$
Indeed, since each $\frac{\xi(v,u,j)}{r_j(v,u)}$ is an exponential random variable with rate ${r_j(v,u)}$, we have that 
\begin{align*}
    \P\left(\frac{\xi(v,v^{-1},j)}{r_j(v,v^{-1})} = \min_{u\sim v}\frac{\xi(v,u,j)}{r_j(v,u)}\right) = \frac{\lambda^{(j)}_v}{\lambda^{(j)}_v-1+\deg(v)},
\end{align*}
and for $w \neq v^{-1}$, $w\sim v$,
\begin{align*}
    \P\left(\frac{\xi(v,w,j)}{r_j(v,w)} = \min_{u\sim v}\frac{\xi(v,u,j)}{r_j(v,u)}\right) = \frac{1}{\lambda^{(j)}_v-1+\deg(v)}.
\end{align*}

For $k\ge 1$, let $$\tau_v^{(k)} := \inf\{n\geq 0: \mathsf{L}_n(v) = k\}$$ be the time of the $k$-th visit to $v$, with the convention $\inf\emptyset=\infty$. In particular, $\tau_{v}^{(M+1)}$ is the first time the random walker is at $v$ and there are no more cookies left. On the event $\{X_n = v, \mathsf{L}_n(v) = j\}$, $j\geq M+1$ the next position is given by
$$X_{n+1} = \arg\min_{u\sim v}\left\{\sum_{k=\mathsf{C}_{\tau_v^{(M+1)}}(v,u)}^{\mathsf{C}_n(v,u)}\xi(v,u,k+M+1)\right\}.$$
By the memoryless property of exponential random variables, we infer that, on the event $\{X_n = v, \mathsf{L}_n(v) = j\}$, with $j\geq M+1$, the process jumps from $v$ to any of its neighbours with uniform probability. 

Hence the process $\mathbf{X} = (X_n)_{n\in\N}$ constructed above is a ${\rm MERW}(\boldsymbol\lambda)$ with the law $\P_{\boldsymbol{\lambda}}$ given by \eqref{P.quenched}.

\subsection{Restrictions and extensions}\label{sec:ext}

Fix a vertex $v\in V\setminus\{\rho\}$. Recall that $\mathcal{P}_v$ denote the shortest path from $\rho$ to the $v$. Let 
\[\delta_n(\mathcal{P}_v) = \inf \left \{k\geq 0: \sum_{j=0}^k \mathds{1} _{\{X_j \in \mathcal{P}_v\}} =n+1\right\} \quad\text{for $n\ge 0$},
\]
with the convention $\inf\emptyset=\infty$, be the subsequence of times when \textbf{X} is on the path $\mathcal{P}_v$. 

Define $m_0 = 0$ and, when $m_k<\infty$, let
$$m_{k+1} = \inf\big\{j > m_k : \delta_j(\mathcal P_v)<\infty,\, X_{\delta_j(\mathcal{P}_v)} \neq X_{\delta_{m_k}(\mathcal{P}_v)}\big\}.$$
Set $m_{k+1}=\infty$ when $m_k=\infty$. Let $$K_v:=\sup\{k\ge 0: m_k<\infty\}.$$
Then $(X_{\delta_{m_k}(\mathcal{P}_v)})_{0\le k\le K_v}$ is the nearest neighbour random walk, which records the trajectory of $\mathbf{X}$ when restricted to $\mathcal{P}_v$ (i.e. the durations when $\X$ is outside $\mathcal{P}_v$ are ignored) and consecutive visits to the same point are removed. The process $(X_{\delta_{m_k}(\mathcal{P}_v)})_{0\le k\le K_v}$ is called \textbf{restriction} of $\X$ on $\mathcal{P}_v$. We also call $K_v$ the \textbf{killing time} of the restriction on $\mathcal{P}_v$.

We next define a nearest neighbour random walk $\X^{(v)} = (X_k^{(v)})_{k\ge0}$ on $\mathcal{P}_v$ as a process coupled with the restriction $(X_{\delta_{m_k}(\mathcal{P}_v)})_{0\le k\le K_v}$ such that, conditional on the event that $\mathcal{P}_v$ is visited infinitely many times by $\X$ (which implies that $K_v=\infty$), the process $\X^{(v)}$ coincides with $(X_{\delta_{m_k}(\mathcal{P}_v)})_{k\ge 0}$. 

We construct $\X^{(v)}$ using the exponential family $\boldsymbol{\xi}$. Let $\mathsf{L}_n^{(v)}(x)$ be the number of visits to $x$ by the process $\mathbf{X}^{(v)}$ and let $\mathsf{C}_n^{(v)}(x,y)$ be the number of crossings from $x$ to $y$ by $\mathbf{X}^{(v)}$ by time $n$.
The process $\mathbf{X}^{(v)}$ is defined as an excited random walk with random cookie stacks, where the number of cookies at a vertex $u\in \mathcal{P}_v$ randomly changes every time the process visits $u$. More specifically, to each $u\in \mathcal{P}_v$ we assign a random sequence $(R^{(v)}_u(k))_{k\ge 0}$ taking values in $\{0,1,\cdots, M\}$, where $R^{(v)}_u(k)$ represents the number of cookies at $u\in\mathcal{P}_v$ just before the $(k+1)$-st visit by $\X^{(v)}$ to $u$. This sequence is defined as follows. Set $R^{(v)}_u(0)=M$. For $0\le i\le M$, let 
$$\mathcal{J}_u(i) := \left\{ \ell\in\{M-i+1,\ldots,M\}: \arg\min_{z\sim u}\dfrac{\xi(u,z,\ell)}{r_\ell(u,z)} \in \mathcal P_v \right\}$$
If $R^{(v)}_u(k)=0$ then set $R_u^{(v)}(k+1)=0$. If  $R^{(v)}_u(k)=i$ with $1\le i\le M$, then set
$$R_u^{(v)}(k+1)=\begin{cases}
M- \inf \mathcal{J}_u(i) & \text{if $\mathcal{J}_u(i)\neq \emptyset$},\\
0, & \text{otherwise}.
\end{cases}
$$
Notice that $(R^{(v)}_u(k))_{k\ge 0}$ is a Markov chain. 

\begin{definition}\label{def:restricted} We define the process $\X^{(v)}=(X^{(v)}_n)_{n\ge 0}$ on $\mathcal{P}_v$ inductively as follows.  First, let $X_0^{(v)} = \rho$. Assume that $(X^{(v)}_k)_{0\le k\le n}$ has been defined. 
The next position $X_{n+1}^{(v)}$ is defined by:
\begin{enumerate}
    \item On the event $\big\{\X^{(v)}_n = u,   R^{(v)}_u(\mathsf{L}^{(v)}_n(u)-1)=m, \mathcal J_u(m)\neq\emptyset\big\}$ for some $1\le m\le M$, set
\[X^{(v)}_{n+1} = \arg\min_{w\sim u: w\in \mathcal{P}_v}  \frac{\xi(u,w,\inf \mathcal{J}_u(m) )}{r_{\inf \mathcal{J}_u(m) }(u,w)}.\]
   \item On the event $\{X^{(v)}_n = u, R_u^{(v)}(\mathsf L_n^{(v)}(u)-1)=m,\, \mathcal J_u(m)=\emptyset\}$ for some $1\le m\le M$, set
\[X_{n+1}^{(v)}
=
\arg\min_{w\sim u:\,w\in\mathcal P_v}
\xi\big(u,w,M+\mathsf{C}_n^{(v)}(u,w)+1\big).\]
\item On the event $\{X^{(v)}_n = u, R^{(v)}_u(\mathsf{L}^{(v)}_n(u)-1)=0\}$, set
\[X^{(v)}_{n+1} = \arg\min_{w\sim u: w \in \mathcal{P}_v}  \left\{\sum_{k = \mathsf C^{(v)}_{\tau_u^*(v)}(u,w)}^{\mathsf C^{(v)}_{n}(u,w)} \xi(u,w,k+M+1)\right\},\]
where
$$\tau^*_u(v) := \inf\left\{n\geq 0 : X_n^{(v)}=u,\, \mathcal J_u\!\left(R_u^{(v)}(\mathsf{L}^{(v)}_n(u)-1)\right)=\emptyset\right\}
$$
is the first time $\X^{(v)}$ visits $u$ when no remaining cookie prescribes a step to a neighbour in $\mathcal P_v$.
\end{enumerate}
We call $\X^{(v)}$ the \textbf{extension} of $\X$ on $\Pcal_{v}$. 
\end{definition}

Random walks with i.i.d. Markovian cookie stacks on $\Z$ have been studied by Kosygina, Peterson and Mountford \cite{KP2017, KMP2022}. We stress that the Markov cookie stacks in their model are assumed i.i.d. across the vertices. However, in our model on a path, the Markov chains $R_u^{(v)}$, with $u\in \mathcal{P}_v$, are independent but not necessarily identically distributed. The transition law of $R_u^{(v)}$ is given as follows. For a fixed $u\in \mathcal{P}_v$ and $1\le j\le M$, let 
\begin{align*}
    p^{(u)}_j:= \mathbb P\left(\arg\min_{z\sim u}\frac{\xi(u,z,j)}{r_j(u,z)} \in \mathcal P_v\right)
    =\begin{cases} \displaystyle
    \frac{\lambda^{(j)}_u + 1}{\lambda^{(j)}_u + \deg(u)-1} =\frac{\alpha_j \deg(u) +2}{(1+\alpha_j)\deg(u)}, & u\in \mathcal{P}_v\setminus\{\rho,v\},\\ \displaystyle
     \frac{1}{\deg(\rho)}, & u=\rho,\\ \displaystyle
\frac{\lambda_v^{(j)}}{\lambda_v^{(j)}+\deg(v)-1}
= \frac{1+\alpha_j\deg(v)}{(1+\alpha_j)\deg(v)}, & u=v.
\end{cases}
\end{align*}
For $0\le i,j\le M$, we have
$$\P\big( R_u^{(v)}(k+1)=j \mid R_u^{(v)}(k)=i\big) =
\begin{cases}
    1, & i=0,\; j=0,\\
    p^{(u)}_{M-j}\displaystyle\prod_{\ell=M-i+1}^{M-j-1}\bigl(1-p^{(u)}_{\ell}\bigr),
    & 1\le i\le M,\; 1\le j<i,\\
    \displaystyle\prod_{\ell=M-i+1}^{M-1}\bigl(1-p^{(u)}_{\ell}\bigr),
    & 1\le i\le M,\; j=0,\\
    0 & \text{otherwise}.
\end{cases}$$
By the above construction, we immediate obtain the following result:

\begin{proposition}[Restriction principle]\label{prop.res}
    We have a.s. 
    $$ X_{\delta_{m_k}(\mathcal{P}_v)}=X^{(v)}_k \quad\text{for all } 0\le k\le K_v.$$
\end{proposition}

\begin{proof}
We prove the assertion by induction on $k$. For $k=0$, both processes start at $\rho$. Assume that the assertion holds up to some $k<K_v$. Since there is no cycle on tree $\mathcal T$, every visit of $\X$ to $\mathcal P_v$ between times $\delta_{m_k}(\mathcal P_v)$ and $\delta_{m_{k+1}}(\mathcal P_v)-1$ is to $X_k^{(v)}$. Let
$u=X_k^{(v)}$ and $n=\delta_{m_{k+1}}(\mathcal P_v)-1$. Thus,
\begin{equation}\label{eq:rest.cross}
X_n=u,\quad \mathsf{C}_n(u,w)=\mathsf{C}_k^{(v)}(u,w)
\quad\text{for }w\sim u,\ w\in\mathcal P_v.
\end{equation}

By the definition of $\mathcal J_u$ and the recursion for $R_u^{(v)}$, both processes use the cookie indices in $\mathcal J_u(M)$ in increasing order. If $\mathsf L_n(u)\le M$, we therefore have
$\mathsf L_n(u) = \inf\mathcal J_u\!\big(R_u^{(v)}(\mathsf L_k^{(v)}(u)-1)\big)$. By the construction in Section~\ref{sec:rubin} and using Case~1 of Definition~\ref{def:restricted}, we have in this case that
$$
X_{n+1}=\arg\min_{\substack{w\sim u\\w\in\mathcal P_v}} \frac{\xi(u,w,\mathsf L_n(u))}{r_{\mathsf L_n(u)}(u,w)} = X_{k+1}^{(v)}.
$$

If $\mathsf L_n(u)\ge M+1$, the recursion for $R_u^{(v)}$ and the definition of $\tau_u^*(v)$ ensure that, for $w\sim u$ with $w\in\mathcal P_v$,
$$
\mathsf{C}_{\tau_u^{(M+1)}}(u,w) = \sum_{\ell=1}^M \mathds{1}_{\left\{ \arg\min_{z\sim u}\frac{\xi(u,z,\ell)}{r_\ell(u,z)}=w \right\}} = \mathsf{C}_{\tau_u^*(v)}^{(v)}(u,w).
$$
Since $X_{n+1}\in\mathcal P_v$, the minimum defining $X_{n+1}$ is restricted to neighbours in $\mathcal P_v$. By \eqref{eq:rest.cross}, we have in this case that
$$
X_{n+1} = \arg\min_{\substack{w\sim u\\w\in\mathcal P_v}} \Bigg\{ \sum_{\ell=\mathsf{C}_{\tau_u^*(v)}^{(v)}(u,w)}^{\mathsf{C}_k^{(v)}(u,w)} \xi(u,w,\ell+M+1) \Bigg\} = X_{k+1}^{(v)}.$$
Here we use Cases~2 and~3 of Definition~\ref{def:restricted}; in Case~2, $\tau_u^*(v)=k$, so each sum consists of a single term.
In both cases,
$$
X_{\delta_{m_{k+1}}(\mathcal P_v)} = X_{n+1} = X_{k+1}^{(v)}.
$$
The assertion follows by induction.
\end{proof}

\section{Markov chains related to the number of backward steps} \label{sec:backward}
In this section we study ruin probabilities along finite paths. Throughout this section, we fix a path $\mathcal P_{v_n}=(v_0,\ldots,v_n)$ from the root $v_0=\rho$ to a vertex $v_n$ at distance $n$ from $\rho$. In Section~\ref{sec:markov}, we represent the numbers of backward steps made by the extension process $\X^{(v_n)}$ before its first visit to $v_n$ by an inhomogeneous Markov chain $Y^{(n)}=(Y_k^{(n)})_{0\le k\le n-1}$. The ruin probability that $\X^{(v_n)}$ reaches $v_n$ before returning to $\rho$ is exactly equal to $\P(Y_{n-1}^{(n)}=0\mid Y_{0}^{(n)}=0)$. The transition laws of $Y^{(n)}$ depend on the degrees of the vertices along $\mathcal P_{v_n}$. In Section~\ref{sec:comp}, we compare $Y^{(n)}$ with a homogeneous Markov chain $\widetilde Y$, whose transition law agrees with that of $Y^{(n)}$ at every state at least $M-1$, where $M$ is the number of cookies initially placed at each vertex. We compute the probability generating functions of $\widetilde Y$ for each initial state and study their asymptotics. Using these estimates and stochastic domination, we prove exact ruin probability asymptotics with exponent $D+1$ in Section~\ref{sec:ruin}, where $D$ is defined in Theorem~\ref{main.thm}. In Section~\ref{sec:gen_ruin}, we also bound the probability of reaching $v_n$ before a prescribed number of returns to $\rho$. These results will be used in the percolation argument in Section~\ref{sec:perc}.

Throughout this paper, we use generic sufficiently small positive constants $c$ and sufficiently large constants $C$, which may change their values between occurrences.

\subsection{Markovian structure of backward steps}\label{sec:markov}
Throughout this section, we fix an infinite self-avoiding path $\varkappa=(v_k)_{k\ge 0}$ with $v_0=\rho$.

Fix $n\ge 1$. We consider the extension process $\X^{(v_n)}$ on the path $\Pcal_{v_n}=(v_k)_{0\le k\le n}$. 
For $1\le k\le n$, let $B_n(k)$ be the number of backward steps of $\X^{(v_n)}$ from $v_k$ to $v_{k-1}$ by the first time when $\X^{(v_n)}$ hits the endpoint $v_n$.
We will show that the sequence
$$
\big(B_n(n), B_n(n-1), \dots, B_n(1)\big)
$$
is an inhomogeneous Markov chain. Note that $B_n(n)=0$. Let $r_n$ be the \textbf{ruin probability} that $\X^{(v_n)}$ reaches $v_n$ before returning to $\rho$. Then
$$r_n=\P(B_n(1)=0).$$

The following coupling of $\boldsymbol{\xi}$ with a collection of Bernoulli random variables $\boldsymbol{\eta} = \left(\eta_x(j) : x\in\varkappa,\ j\ge1\right)$ will be useful when analysing the behaviour of the sequence $(B_n(n-k))_{0\le k\le n-1}$. For $x\in\varkappa$, let $x^{+1}$ denote its unique child in $\varkappa$. We use the construction of $\mathcal J_x$ in Section~\ref{sec:ext}. Recall that
$$
\mathcal J_x(M)= \left\{ \ell\in\{1,\ldots,M\}: \arg\min_{z\sim x}\frac{\xi(x,z,\ell)}{r_\ell(x,z)} \in\{x^{-1},x^{+1}\} \right\}.
$$
For $1\le j\le\operatorname{card}(\mathcal J_x(M))$, let $\ell$ be the $j$-th smallest element of $\mathcal J_x(M)$ and define
$$
\eta_x(j) = \mathds{1}_{\left\{ \arg\min_{z\sim x}\frac{\xi(x,z,\ell)}{r_\ell(x,z)} =x^{+1} \right\}}.
$$
For $j>\operatorname{card}(\mathcal J_x(M))$, define recursively
$$
\eta_x(j) = \mathds{1}_{\left\{
\sum_{k=\sum_{i=1}^{|\mathcal J_x(M)|}\eta_x(i)}^{\sum_{i=1}^{j-1}\eta_x(i)} \xi(x,x^{+1},k+M+1)<\sum_{k=\sum_{i=1}^{|\mathcal J_x(M)|}(1-\eta_x(i))}^{\sum_{i=1}^{j-1}(1-\eta_x(i))}\xi(x,x^{-1},k+M+1)\right\}}.
$$
We also set $\eta_\rho(j)=1$ for every $j\ge1$. 
Notice that, for every $x\in\mathcal P_{v_n}\setminus\{\rho,v_n\}$ and $j\ge1$, we also have
$$
\eta_x(j) = \sum_{k=0}^{\infty} \mathds{1}_{\{X_k^{(v_n)}=x,\ X_{k+1}^{(v_n)}=x^{+1},\ \mathsf L_k^{(v_n)}(x)=j\}}.
$$
Moreover, for each vertex $x\in\varkappa\setminus\{\rho\}$, we have
$\sigma(\eta_x(j):j\ge1) \subseteq \sigma(\xi(x,z,\ell):z\sim x,\ \ell\ge1).$
Since the random variables in $\boldsymbol\xi$ are independent, the sequences $(\eta_x(j))_{j\ge1}$ are independent for distinct vertices $x\in \varkappa\setminus\{\rho\}$.

For integers $m\ge0$ and $k\ge 1$, let
\begin{align}\label{def.Phi}\Phi_k(m) = \inf\left\{\ell\ge0:\sum_{j=1}^{\ell+m}\eta_{v_k}(j)=m\right\}\end{align}
which, for $m\ge1$, can be viewed as the number of backward steps from $v_k$ before the $m$-th forward step from $v_k$.
Since the sequences $(\eta_{v_k}(j))_{j\ge1}$ are independent for distinct $k\ge 1$, the random functions $(\Phi_k)_{k\ge1}$ are independent. The law of $\Phi_k$ depends only on $\deg(v_k)$ and $(\lambda_{v_k}^{(j)})_{j=1}^M$.
Since every edge on $\mathcal P_{v_n}$ is crossed once more forward than backward before the first hitting time of $v_n$, we have
\begin{equation}\label{eq:B.recursion}
B_n(k) = \Phi_{k}\bigl(B_n(k+1)+1\bigr)
\quad\text{for } 1\le k\le n-1.
\end{equation}

For each $n\ge1$, let $(Y_m^{(n)})_{0\le m\le n-1}$ be an inhomogeneous Markov chain, whose transition law is given by
\begin{equation}\label{def:Y}
\mathbb P_i\bigl(Y_{m+1}^{(n)}=j\mid Y_m^{(n)}=\ell\bigr) = \mathbb P\bigl(\Phi_{n-m-1}(\ell+1)=j\bigr),
\end{equation}
for $0\le m\le n-2$ and $\ell,j\in\mathbb Z_+$. Here $\mathbb P_i$ denotes its law with initial state $i$.  Under $\mathbb P_0$, we have
\begin{equation}\label{eq:comp}
\left(B_n(n),B_n(n-1),\dots,B_n(1)\right)
\overset{\mathrm d}{=}
\left(Y_0^{(n)},Y_1^{(n)},\dots,Y_{n-1}^{(n)}\right).
\end{equation}
Consequently, we have
\begin{align}\label{def:ruin} r_n=\P(B_n(1)=0)=\P_0(Y_{n-1}^{(n)}=0).\end{align}

\subsection{Comparison with a time-homogeneous Markov chain}\label{sec:comp}
Fix $n\ge 1$. Recall that $(Y_k^{(n)})_{0\le k\le n-1}$ is the inhomogeneous Markov chain defined by \eqref{def:Y}. 

We aim to analyse $\P_0(Y_{n-1}^{(n)}=0)$, which is equal to the ruin probability of the extension process $\X^{(v_n)}$ defined on $\mathcal{P}_{v_n}$. To this end, we compare $Y^{(n)}$ with a time-homogeneous Markov chain, $\widetilde Y$.

In this section, we assume that, for each $u\in\varkappa\setminus\{\rho\}$ and $1\le j\le M$,
$$\lambda_u^{(j)}=1+\alpha_j\deg(u),$$ and set
$$\beta_j:=\frac{\alpha_j}{1+\alpha_j}.$$
Recall that for $m\ge 0$ and $k\ge 1$, $\Phi_k(m)$ is defined by \eqref{def.Phi}.

\begin{lemma}\label{lem:F_pgf}
For $1\le k\le n-1$, $0\le s<2$ and $m\ge M$,
$$
\mathbb E[s^{\Phi_k(m)}]
=
\left(\frac{1}{2-s}\right)^m
\prod_{j=1}^M \bigl(1-\beta_j(1-s)\bigr).
$$
\end{lemma}

\begin{proof}
For $1\le k\le n-1$, let $$A_k^i=\mathds{1}_{\left\{\arg\min_{z\sim v_k}\frac{\xi(v_k,z,i)}{r_i(v_k,z)}=v_{k+1}\right\}}
\quad\text{and}\quad B_k^i=\mathds{1}_{\left\{\arg\min_{z\sim v_k}\frac{\xi(v_k,z,i)}{r_i(v_k,z)}=v_{k-1}\right\}}$$
be the indicators that, on the $i$-th visit to $v_k$, the walk $\X$ makes a forward step to $v_{k+1}$, and a backward step to $v_{k-1}$ respectively. Notice that it is possible that $A_k^i=B_k^i=0$, which corresponds to the event that the $i$-th jump by $\X$ from $v_k$ is to a child of $v_k$ not lying in $\varkappa$. Define 
$$A_k=\sum_{i=1}^M A_k^i,\quad B_k=\sum_{i=1}^M B_k^i.$$
Then $A_k$ (and $B_k$) are the number of forward steps (backward steps respectively) along $\varkappa$ during the first $M$ visits to $v_k$. For $m\ge M$, we notice that
$$
\Phi_k(m)=B_k+H_{m-A_k},
$$
where $H_r$ is the number of backward steps before the $r$-th forward step after the cookies have been consumed, which is negative binomial distributed, counting the number of failures before the $r$-th success in a sequence of independent $\operatorname{Bernoulli}(1/2)$ trials. In the expression above, when $m\le A_k$, we use the convention $H_{m-A_k}=0$.

Conditioning on $(A_k, B_k)$ with the law of total probability and using the fact that 
$H_{m-i}$ is negative binomial with parameter $1/2$, we have
\begin{align*}
    \E\left[s^{\Phi_k(m)}\right]
    = \E\left[s^{B_k + H_{m-A_k}}\right]
    = \sum_{i=0}^M \sum_{j=0}^{M-i} \P(A_k = i,\, B_k = j)\, s^j \left(\frac{1}{2-s}\right)^{m-i}.
\end{align*}
Factoring out $(2-s)^{-m}$ and recognising the remaining sum as an expectation, we have
\begin{align*}
    \E\left[s^{\Phi_k(m)}\right] = \left(\frac{1}{2-s}\right)^{m} \E\left[s^{B_k}(2-s)^{A_k}\right].
\end{align*}
Since the pairs $(A_k^i, B_k^i)$ are independent across $i$,
\begin{align*}
    \E\left[s^{B_k}(2-s)^{A_k}\right]
    = \prod_{i=1}^M \E\left[s^{B_k^i}(2-s)^{A_k^i}\right].
\end{align*}
For each $1\le i\le M$, we have
\begin{align*}
&\P(A_k^i=1, B_k^i=0)=\frac{1}{\lambda_{v_k}^{(i)}+\deg(v_k)-1}, \quad
\P(A_k^i=0, B_k^i=1)=\frac{\lambda_{v_k}^{(i)}}{\lambda_{v_k}^{(i)}+\deg(v_k)-1},\\
&\P(A_k^i=B_k^i=0)=\frac{\deg(v_k)-2}{\lambda_{v_k}^{(i)}+\deg(v_k)-1}.
\end{align*}
Therefore
\begin{align*}
    \E\left[s^{B_k^i}(2-s)^{A_k^i}\right]
    &= \frac{2-s}{\lambda_{v_k}^{(i)}+\deg(v_k)-1}
    + \frac{\lambda_{v_k}^{(i)}s}{\lambda_{v_k}^{(i)}+\deg(v_k)-1}
    + \frac{\deg(v_k)-2}{\lambda_{v_k}^{(i)}+\deg(v_k)-1}  \\
    &=   \frac{\deg(v_k)+(\lambda_{v_k}^{(i)}-1)s}{\lambda_{v_k}^{(i)}+\deg(v_k)-1}.
\end{align*}
Using the assumption $\lambda_{v_k}^{(i)}=1+\alpha_i\deg(v_k)$, this becomes
\begin{align*}
    \E\left[s^{B_k^i}(2-s)^{A_k^i}\right]  &=
    \frac{\deg(v_k)+\alpha_i\deg(v_k)s}{\deg(v_k)+\alpha_i\deg(v_k)} = 1-\frac{\alpha_i}{1+\alpha_i}(1-s) = 1-\beta_i(1-s).
\end{align*}
Hence
\begin{align*}
\E\left[s^{B_k}(2-s)^{A_k}\right]  = \prod_{i=1}^M \bigl(1-\beta_i(1-s)\bigr).
\end{align*}
Substituting back, we obtain the desired result.
\end{proof}

Assume that $\alpha_j\ge0$ for $1\le j\le M$. We now define the homogeneous Markov chain $\widetilde Y$.

\begin{definition}\label{def:Ytilde}
Let $(\widetilde Y_k)_{k\ge 0}$ on $\Z_+$ be the Markov chain such that for each $m\ge 0$,
    $$\mathbb E\!\left[s^{\widetilde Y_{k+1}}\,\middle|\,\widetilde Y_k=m\right]=
(2-s)^{-(m+1)}\prod_{j=1}^M\bigl(1-\beta_j(1-s)\bigr).$$
\end{definition}
By \eqref{def:Y} and Lemma \ref{lem:F_pgf}, we notice that for $0\le k\le n-2$, the conditional laws of $Y_{k+1}$ given $Y_k=m$ and $\widetilde Y_{k+1}$ given $\widetilde Y_k=m$ are identical when $m\ge M-1$.

Let $\widetilde{\mathbb P}_m$ be the probability measure conditioning on the event $\{\widetilde Y_0=m\}$.  Define 
$$G_n(m,s):=\widetilde{\mathbb E}_m[s^{\widetilde Y_n}],
\quad m\in\mathbb Z_+,\ 0\le s\le 1.$$ 


\begin{lemma}\label{lem.G}
For each fixed $m\in \Z_+$, $G_n(m,s)$ satisfies the functional equation
\begin{equation}\label{recur.g}
G_{n+1}(m,s)= \frac{1}{2-s} \prod_{j=1}^M \bigl(1-\beta_j(1-s)\bigr) \cdot G_n\Bigl(m,\frac{1}{2-s}\Bigr)
\end{equation}
with the initial condition
$G_0(m,s) =s^m.$
Additionally, $G_n(m,s)$ has the exact solution 
\begin{equation}\label{exact.g}
G_n(m,s) = \frac{\bigl(n-(n-1)s\bigr)^m}{\bigl((n+1)-ns\bigr)^{m+1}}  \prod_{i=0}^{n-1}\prod_{j=1}^M   \left(1 - \beta_j\left(1-\frac{i-(i-1)s}{(i+1)-is}\right)\right).
\end{equation}
\end{lemma}

\begin{proof}
Since $\widetilde Y_0=m$ under $\widetilde\P_m$, we have $G_0(m,s)=s^m$. By Definition~\ref{def:Ytilde} and conditioning on $\widetilde Y_n$, we obtain
\begin{align*}
G_{n+1}(m,s)
&=\sum_{k=0}^{\infty}\widetilde\P_m(\widetilde Y_n=k)
(2-s)^{-(k+1)}
\prod_{j=1}^M\bigl(1-\beta_j(1-s)\bigr)\\
&=\frac{1}{2-s}\prod_{j=1}^M\bigl(1-\beta_j(1-s)\bigr)
G_n\Bigl(m,\frac{1}{2-s}\Bigr),
\end{align*}
which proves \eqref{recur.g}.

We prove \eqref{exact.g} by induction. The case $n=0$ agrees with the initial condition, with the empty product equal to $1$. Assume that \eqref{exact.g} holds for $n$. We have
$$
\frac{i-(i-1)\frac{1}{2-s}}{(i+1)-i\frac{1}{2-s}}
=\frac{i+1-is}{(i+2)-(i+1)s}
$$
and
$$
\frac{1}{2-s}
\frac{\left(n-(n-1)\frac{1}{2-s}\right)^m}
{\left((n+1)-n\frac{1}{2-s}\right)^{m+1}}
=
\frac{\bigl((n+1)-ns\bigr)^m}
{\bigl((n+2)-(n+1)s\bigr)^{m+1}}.
$$
Substituting \eqref{exact.g} into \eqref{recur.g} and using these identities, we re-index the product from $i=0,\ldots,n-1$ to $i=1,\ldots,n$. The factor $\prod_{j=1}^M(1-\beta_j(1-s))$ in \eqref{recur.g} is the $i=0$ factor. Therefore
$$
G_{n+1}(m,s)
=
\frac{\bigl((n+1)-ns\bigr)^m}
{\bigl((n+2)-(n+1)s\bigr)^{m+1}}
\prod_{i=0}^{n}\prod_{j=1}^M
\left(1-\beta_j\left(1-\frac{i-(i-1)s}{(i+1)-is}\right)\right),
$$
which completes the induction.
\end{proof}

Recall that the Gamma function is defined, for ${\rm Re}(x)>0$, by
$$
\Gamma(x):=\int_0^\infty t^{x-1}e^{-t}\,dt.
$$
For ${\rm Re}(z)>-1$, we use $1/\Gamma(z)=z/\Gamma(z+1)$, including $1/\Gamma(0)=0$.
In the next result, we obtain the exact asymptotic behaviour of $G_n(m,s)$ as $n\to\infty$.

\begin{lemma}\label{lem:G.asym} For each fixed $m\in\Z_+$ and $s\in[0,1)$, we have
$$G_n(m,s)\sim\frac{C(s)}{n^{1+D}}\quad\text{as }n\to\infty,$$
where
$$C(s):=\frac{1}{1-s}\prod_{j=1}^M\frac{\Gamma((1-s)^{-1})}{\Gamma((1-s)^{-1}-\beta_j)}.$$
Moreover, for every $m\in\Z_+$, $s\in[0,1)$ and $n\ge1$,
$$G_n(m,s)\le(1+n(1-s))^{-D-1}\le (1-s)^{-D-1} n^{-D-1}$$
Additionally, for each fixed $m\in\Z_+$,
$$G_n(m,s)=\frac{1}{n^{1+D}}\left(C(s)+\epsilon_n(m,s)\right),\quad |s|=\frac12,$$
where $\sup_{s\in \C: |s|=1/2}|\epsilon_n(m,s)|\rightarrow 0$ as $n \rightarrow \infty$.
\end{lemma}

\begin{proof}
We first note that
\begin{align}\label{term1}
\frac{\bigl(n-(n-1)s\bigr)^m}{\bigl((n+1)-ns\bigr)^{m+1}}=\frac{\bigl(1+(n-1)(1-s)\bigr)^m}{\bigl(1+n(1-s)\bigr)^{m+1}}\sim\frac{1}{n(1-s)}\quad\text{as }n\to\infty,
\end{align}
for any fixed $m\in\Z_+$ and any fixed $s\in[0,1)$. To analyse the product term in \eqref{exact.g}, we write
\begin{align*}
\prod_{i=0}^{n-1}\prod_{j=1}^M\left(1-\beta_j\left(1-\frac{i-(i-1)s}{(i+1)-is}\right)\right)=\prod_{i=0}^{n-1}\prod_{j=1}^M\left(1-\beta_j\frac{1-s}{1+i(1-s)}\right).
\end{align*}
For each $j=1,\ldots,M$, we have
\begin{align*}
\prod_{i=0}^{n-1} \left(1-\beta_j\frac{1-s}{1+i(1-s)}\right)
&= \prod_{i=0}^{n-1}
\frac{i+(1-s)^{-1}-\beta_j}{i+(1-s)^{-1}}\\
&= \frac{\Gamma(n+(1-s)^{-1}-\beta_j)} {\Gamma(n+(1-s)^{-1})}\frac{\Gamma((1-s)^{-1})} {\Gamma((1-s)^{-1}-\beta_j)}\\
&\sim n^{-\beta_j} \frac{\Gamma((1-s)^{-1})}
{\Gamma((1-s)^{-1}-\beta_j)} \quad\text{as }n\to\infty.
\end{align*}
where the last line follows from Lemma~\ref{lem:Gamma} (see the appendix).
Multiplying over $1\le j\le M$ gives us
\begin{equation}\label{term2}
\prod_{i=0}^{n-1}\prod_{j=1}^M
\left(1-\beta_j\frac{1-s}{1+i(1-s)}\right)
\sim
n^{-D}
\prod_{j=1}^M
\frac{\Gamma((1-s)^{-1})}
{\Gamma((1-s)^{-1}-\beta_j)}.
\end{equation}
Combining \eqref{exact.g}, \eqref{term1} and \eqref{term2}, we obtain the first assertion.

For the upper bound, when $s\in[0,1)$, we have
$$
\frac{\bigl(n-(n-1)s\bigr)^m}
{\bigl((n+1)-ns\bigr)^{m+1}}
=
\frac{1}{1+n(1-s)}
\left(1-\frac{1-s}{1+n(1-s)}\right)^m
\le\frac{1}{1+n(1-s)}.
$$
Using $1-u\le e^{-u}$, we also have
\begin{align*}
\prod_{i=0}^{n-1}\prod_{j=1}^M
\left(1-\beta_j\frac{1-s}{1+i(1-s)}\right)
&\le
\exp\left(-D\sum_{i=0}^{n-1}
\frac{1-s}{1+i(1-s)}\right)\\
&\le
\exp\left(-D\int_0^n
\frac{1-s}{1+t(1-s)}\,dt\right)\\
&=(1+n(1-s))^{-D}.
\end{align*}
Combining these bounds with \eqref{exact.g}, we obtain
$$
G_n(m,s)\le(1+n(1-s))^{-D-1}
\le(1-s)^{-D-1}n^{-D-1}.
$$

Finally, we prove uniform convergence for $s\in\C$ with $|s|=1/2$. It is sufficient to prove that both \eqref{term1} and \eqref{term2} have uniform errors for $|s|=1/2$. For \eqref{term1}, we note that
$$\frac{\bigl(n-(n-1)s\bigr)^m}{\bigl((n+1)-ns\bigr)^{m+1}}n(1-s)=\left(1-\frac{1}{1+n(1-s)}\right)\left(1-\frac{1-s}{1+n(1-s)}\right)^m.$$
Since $|s|=1/2$, we have ${\rm Re}(1-s)\ge1/2$ and $|1-s|\le3/2$. Therefore,
$$\left|\frac{1}{1+n(1-s)}\right|\le\frac{1}{1+n/2}\quad\text{and}\quad\left|\frac{1-s}{1+n(1-s)}\right|\le\frac{3}{n+2}.$$
Thus, for each fixed $m\in\Z_+$,
\begin{align}\label{asym1}
    \sup_{|s|=1/2}\left|\frac{\bigl(n-(n-1)s\bigr)^m}{\bigl((n+1)-ns\bigr)^{m+1}}n(1-s)-1\right|\rightarrow0\quad\text{as }n\to\infty.
\end{align}
Applying Lemma~\ref{lem:Gamma} with $a=(1-s)^{-1}-\beta_j$ and $b=(1-s)^{-1}$, we have, uniformly for $|s|=1/2$,
\begin{align}\label{asym2}
    \frac{\Gamma(n+(1-s)^{-1}-\beta_j)}{\Gamma(n+(1-s)^{-1})}=n^{-\beta_j}(1+O(n^{-1}))\quad\text{as }n\to\infty.
\end{align}
The identity also holds for $\beta_j=0$. 

Since ${\rm Re}((1-s)^{-1})>0$ for $|s|=1/2$, the function
${\Gamma((1-s)^{-1})}/{\Gamma((1-s)^{-1}-\beta_j)}$
is continuous and bounded on $|s|=1/2$. Therefore,
\begin{align}\label{asym3}
    \sup_{|s|=1/2}\left|n^D\prod_{i=0}^{n-1}\prod_{j=1}^M\left(1-\beta_j\frac{1-s}{1+i(1-s)}\right)-\prod_{j=1}^M\frac{\Gamma((1-s)^{-1})}{\Gamma((1-s)^{-1}-\beta_j)}\right|\rightarrow0\quad\text{as }n\to\infty.
\end{align}
Combining this with the uniform convergence in \eqref{asym1} we obtain
$$\sup_{|s|=1/2}|n^{1+D}G_n(m,s)-C(s)|\rightarrow0 \quad\text{as }n\to\infty.$$ 
This completes the proof. 
    
\end{proof}

In order to compare $Y^{(n)}$ and $\widetilde{Y}$,
we need to carefully analyses when $Y^{(n)}$ visits small states in $$S:=\{0,1,2,\cdots, M-2\}.$$

\begin{lemma}\label{lem:small_state2}
Let $U:=\{M-1,M,\dots,2M-1\}\subset S^c$. There exists a constant $\delta_*>0$, depending only on $M$ and $(\alpha_j)_{j=1}^M$, such that for every $0\le k\le n-2$ and every $m\in S$,
$$\mathbb P(Y_{k+1}\in U\mid Y_k=m)\ge\delta_*.$$
\end{lemma}

\begin{proof}
By \eqref{def:Y}, conditionally on $\{Y_k=m\}$, $Y_{k+1}$ has the same law as $\Phi_{n-k-1}(m+1)=B_{n-k-1}+H_{m+1-A_{n-k-1}},$ where we recall from the proof of Lemma~\ref{lem:F_pgf} that $A_{k}$ and $B_{k}$ are the number of forward steps from $v_k$ to $v_{k+1}$ and backward steps from $v_k$ to $v_{k-1}$, respectively, prescribed by the first $M$ cookies at $v_{k}$.  In particular,
$$A_{k}=\sum_{j=1}^M A_{k}^j,$$
where, for $1\le j\le M$, $A_{k}^j$ is the indicator that the $j$-th cookie at $v_k$ prescribes a forward step from $v_{k}$ to $v_{k+1}$.
For each $1\le j\le M$, we have
$$\mathbb P(A_{k}^j=1)=\frac{1}{\lambda_{v_{k}}^{(j)}+\deg(v_{k})-1}=\frac{1}{(1+\alpha_j)\deg(v_{k})}.$$
Since $\alpha_j\ge0$ and $\deg(v_{n-k-1})\ge2$, we have
$\mathbb P(A_{n-k-1}^j=0)\ge 1/2.$ The random variables $(A_{n-k-1}^j)_{1\le j\le M}$ are independent. Therefore
$$\mathbb P(A_{n-k-1}=0)\ge2^{-M}.$$
On the event $\{A_{n-k-1}=0\}$, we have
$$\Phi_{n-k-1}(m+1)=B_{n-k-1}+H_{m+1}.$$
Moreover, $0\le B_{n-k-1}\le M$. The event $\{H_{m+1}=M-1\}$ has probability bounded below uniformly over $m\in S$. Indeed, one way for this event to occur is that the first $M-1$ trials are failures and the next $m+1$ trials are successes. This event has probability $2^{-(M-1)}2^{-(m+1)}=2^{-(M+m)}\ge2^{-(2M-2)}$, as $m\le M-2$.
The random variable $H_{m+1}$ is independent of $(A_{n-k-1},B_{n-k-1})$. Hence
$$\mathbb P(A_{n-k-1}=0,\ H_{m+1}=M-1)\ge2^{-M}2^{-(2M-2)}=2^{-(3M-2)}.$$

On the event $\{A_{n-k-1}=0,\ H_{m+1}=M-1\}$, we have $\Phi_{n-k-1}(m+1)=B_{n-k-1}+M-1$. Since $0\le B_{n-k-1}\le M$, on this event we have
$$M-1\le\Phi_{n-k-1}(m+1)\le2M-1,$$
and thus $\Phi_{n-k-1}(m+1)\in U$. Therefore, for every $0\le k\le n-2$ and every $m\in S$,
$$\mathbb P(Y_{k+1}\in U\mid Y_k=m)=\mathbb P(\Phi_{n-k-1}(m+1)\in U)\ge2^{-(3M-2)}:=\delta_*.$$

\end{proof}

Let
$$\widetilde T_S:=\inf\{k\ge0:\widetilde Y_k\in S\}\quad\text{and}\quad T_S^+:=\inf\{1\le k\le n-1:Y_k\in S\}$$
be respectively the first entrance time of $\widetilde Y$ to $S$ and the first return time of $Y$ to $S$ before time $n$.

\begin{lemma}\label{lem:escape} Assume that $D>0$.
For every $x\in S^c$, we have
$$
\widetilde{\mathbb P}_x(\widetilde T_S=\infty)>0.
$$
\end{lemma}
\begin{proof}
Fix $a\ge1$ and $q\in(0,D)$, and define $V_n:=(\widetilde Y_n+a)^{-q}$. Fix some $R\ge M-1$ to be chosen sufficiently large, and let
$$\zeta_R:=\inf\{n\ge0:\widetilde Y_n\le R\}.$$
Using the identity $\frac{1}{(x+a)^q}=\frac{1}{\Gamma(q)}\int_0^\infty t^{q-1}e^{-t(x+a)}\rmd t$, we can write
$$\E[V_{n+1}\mid\widetilde Y_n=y]=\frac{1}{\Gamma(q)}\int_0^\infty t^{q-1}e^{-at}\E[e^{-t\widetilde Y_{n+1}}\mid\widetilde Y_n=y]\rmd t.$$
Recall from Definition~\ref{def:Ytilde} that
$$\mathbb E\!\left[s^{\widetilde Y_{n+1}}\,\middle|\,\widetilde Y_n=m\right]=(2-s)^{-(m+1)}\prod_{j=1}^M\bigl(1-\beta_j(1-s)\bigr).$$
Evaluating this at $s=e^{-t}$ gives us
$$\E[V_{n+1}\mid\widetilde Y_n=y]=\frac{1}{\Gamma(q)}\int_0^\infty t^{q-1}e^{-at}(2-e^{-t})^{-(y+1)}\prod_{j=1}^M(1-\beta_j(1-e^{-t}))\rmd t.$$

By Lemma \ref{lem:lypint} (see the appendix), 
we may choose $R\ge M-1$ sufficiently large such that, for every $y>R$,
$$\E[V_{n+1}\mid\widetilde Y_n=y]\le(y+a)^{-q}.$$
Hence,
$$\E[V_{n+1}-V_n\mid\widetilde Y_n=y]\le0,\quad y>R.$$

We have shown that $(V_{n\wedge\zeta_R})_{n\ge0}$ is a non-negative supermartingale. Setting $m>R$ and using the optional stopping theorem, we get
$$\E_m[V_{n\wedge\zeta_R}]\le\E_m[V_0]=\left(\frac{1}{m+a}\right)^q.$$
Note that on $\{\zeta_R\le n\}$, we have $V_{n\wedge\zeta_R}=V_{\zeta_R}\ge\left(\frac{1}{R+a}\right)^q$, so
$$\left(\frac{1}{R+a}\right)^q\P_m(\zeta_R\le n)\le\E_m[V_{n\wedge\zeta_R}]\le\E_m[V_0]=\left(\frac{1}{m+a}\right)^q.$$
Rearranging this gives
$$\P_m(\zeta_R\le n)\le\left(\frac{R+a}{m+a}\right)^q.$$
Taking $n$ to infinity and using the continuity of probability measures gives us
$$\widetilde{\mathbb P}_m(\widetilde T_S<\infty)\le\widetilde{\mathbb P}_m(\zeta_R<\infty)\le\left(\frac{R+a}{m+a}\right)^q<1,$$
where the first inequality holds since $R\ge M-1$ and $S=\{0,\ldots,M-2\}$. In other words, when $m>R$, we have
$$\widetilde{\mathbb P}_m(\widetilde T_S=\infty)>0.$$
To extend this to $x\in S^c$ with $x\le R$, choose $m>R$. By the transition law of $\widetilde Y$, we have $\widetilde{\mathbb P}_x(\widetilde Y_1=m)>0$. Using the Markov property, we obtain
$$\widetilde{\mathbb P}_x(\widetilde T_S=\infty)\ge\widetilde{\mathbb P}_x(\widetilde Y_1=m)\widetilde{\mathbb P}_m(\widetilde T_S=\infty)>0,$$
completing the proof for all initial states.
\end{proof}

For $i\ge 0$, $j\in S$ and $n\ge0$, set
$$
L_j^{(i)}(n):=\widetilde{\mathbb P}_i(\widetilde Y_n=j).
$$
In particular, $L_j^{(i)}(0)=\mathds1_{\{i=j\}}$. Also define
$$
\widehat L_{0,j}(z):=\sum_{n=1}^{\infty}L_j^{(0)}(n)z^n
\quad\text{for }|z|<1.
$$

\begin{lemma}\label{prop:first-hit}
Assume that $M\ge2$ and $D>0$. There exists a constant $C<\infty$, depending only on $M$ and $(\alpha_k)_{1\le k\le M}$, such that
\begin{equation}\label{eq:L.sbound}
\sum_{j\in S}L_j^{(i)}(n)\le C(n+1)^{-D-1}
\quad\text{for every }i\ge 0,\ n\ge0.
\end{equation}
For each $j\in S$, there exists $\ell_j\ge0$ such that
\begin{equation}\label{eq:Lasym}
L_j^{(0)}(n)\sim \ell_j n^{-D-1}
\quad\text{as }n\rightarrow\infty,
\end{equation}
where
$\ell_0=\prod_{i=1}^M\frac{1}{\Gamma(1-\beta_i)}>0.$
Moreover,
\begin{equation}\label{eq:L.taillim}
\sum_{m=n}^{\infty}L_j^{(0)}(m)\sim \frac{\ell_j}{D} n^{-D}
\quad\text{as }n\rightarrow\infty.
\end{equation}
\end{lemma}

\begin{proof}
For any fixed $s\in(0,1)$, we have
$$
s^{M-2}\sum_{j\in S}L_j^{(i)}(n)
\le\sum_{j\in S}L_j^{(i)}(n)s^j
\le G_n(i,s).
$$
By Lemma~\ref{lem:G.asym}, we have
$$
G_n(i,s)\le(1+n(1-s))^{-D-1}
\quad\text{for every }i\ge 0,\ n\ge1.
$$
Also note that $\sum_{j\in S}L_j^{(i)}(0)\le1$. We therefore obtain \eqref{eq:L.sbound} with some sufficiently large constant $C$.

For $\mathcal C=\{s:|s|=1/2\}$, applying Cauchy's integral formula, we have
$$
L_j^{(0)}(n)=\frac{1}{2\pi i}\oint_{\mathcal C}\frac{G_n(0,s)}{s^{j+1}}\rmd s.
$$
By Lemma~\ref{lem:G.asym},
$$
G_n(0,s)=n^{-D-1}\bigl(C(s)+\epsilon_n(0,s)\bigr),
\quad
\sup_{s\in\mathcal C}|\epsilon_n(0,s)|\rightarrow0.
$$
Since
$$
\left|\frac{1}{2\pi i}\oint_{\mathcal C}\frac{\epsilon_n(0,s)}{s^{j+1}}\rmd s\right|
\le2^j\sup_{s\in\mathcal C}|\epsilon_n(0,s)|\rightarrow0,
$$
we obtain \eqref{eq:Lasym}, with
$$
\ell_j=\frac{1}{2\pi i}\oint_{\mathcal C}\frac{C(s)}{s^{j+1}}\rmd s.
$$
Each $\ell_j$ is non-negative as $L_j^{(0)}(n)\ge0$. In particular,
$$
\ell_0=C(0)=\prod_{i=1}^M\frac{1}{\Gamma(1-\beta_i)}>0.
$$
Finally, since $n^D\sum_{m=n}^{\infty}m^{-D-1}\rightarrow1/D$, by \eqref{eq:Lasym} we obtain \eqref{eq:L.taillim}.
\end{proof}

\subsection{Ruin probability}\label{sec:ruin}

Fix an infinite self-avoiding path $\varkappa=(v_k)_{k\ge 0}$ with $v_0=\rho$. Recall that $r_n$ is the ruin probability that $\X^{(v_n)}$ reaches $v_n$ before returning to $\rho$, and that $(Y_k^{(n)})_{0\le k\le n-1}$ is the inhomogeneous Markov chain defined by \eqref{def:Y}. By \eqref{def:ruin}, we have
$$r_n=P_0(Y_{n-1}^{(n)}=0).$$
In this subsection we prove the following exact asymptotic formula of the ruin probability $r_n$ as $n\to\infty$.

\begin{proposition}\label{thm:ruin_prob}
     There exists a constant $c_*(\varkappa)\in (0,\infty)$ depend only on the ray $\varkappa$ such that
    $$
r_n= \mathbb P_0(Y_{n-1}^{(n)}=0)\sim c_*(\varkappa) n^{-(D+1)} \quad {\rm as  }\ n\to\infty. 
$$
Furthermore, there exists a constant $C\in [1,\infty)$ depending only on $M$ and $(\alpha_k)_{1\le k\le M}$  such that 
\begin{equation}\label{eq:ruin.uniform}
C^{-1}n^{-D-1}
\le r_n
\le Cn^{-D-1}.
\end{equation} uniformly for all ray $\varkappa$. 
\end{proposition}

We first prove the monotonicity of the ruin probabilities. 

\begin{lemma}\label{lem:monotone}
 The probability $r_n=\mathbb P_0(Y_{n-1}^{(n)}=0)$ is eventually non-increasing in $n$.
\end{lemma}
\begin{proof}
We use the coupling with the Bernoulli random variables $(\eta_{v_k}(j))_{j\ge1,\ k\ge1}$ defined in Section~\ref{sec:markov} to show that $r_{n+1}\le r_n$ for all $n\ge1$.

Recall that $B_n(n)=0$ and, by \eqref{eq:B.recursion}, we have
$$B_{n+1}(n)=\Phi_n(1)\ge0=B_n(n).$$
Since each $\Phi_k$ is non-decreasing, we obtain successively, for $k=n-1,\dots,1$,
$$B_{n+1}(k)=\Phi_k(B_{n+1}(k+1)+1)\ge\Phi_k(B_n(k+1)+1)=B_n(k).$$
By \eqref{eq:comp}, we therefore have
$$r_{n+1}=\mathbb P(B_{n+1}(1)=0)\le\mathbb P(B_n(1)=0)=r_n.$$
\end{proof}

For each integer $d\ge2$, let $Q_d$ be the transition matrix in \eqref{def:Y} corresponding to a vertex of degree $d$, with parameters
$\lambda^{(j)}=1+\alpha_jd$. Thus, for $k\ge1$ and $i,j\in\mathbb Z_+$,
$$Q_{\deg(v_k)}(i,j)=\mathbb P\bigl(\Phi_k(i+1)=j\bigr).$$
Also set $\widetilde Q(i,j):=\widetilde{\mathbb P}_i(\widetilde Y_1=j)$.
By Lemma~\ref{lem:F_pgf} and Definition~\ref{def:Ytilde}, we have
\begin{equation}\label{eq:large.state}
Q_d(i,j)=\widetilde Q(i,j) \quad\text{for }d\ge2,\ i\ge M-1,\ j\ge 0.
\end{equation}


\begin{lemma}\label{lem:ruinhomo}
Assume that $M\ge2$ and $D>0$. There exists a constant $C<\infty$, depending only on $M$ and $(\alpha_\ell)_{1\le\ell\le M}$, such that
\begin{equation}\label{eq:Q2bound}
Q_2^n(0,0)\le C(n+1)^{-D-1}
\quad\text{for every }n\ge0.
\end{equation}
\end{lemma}

\begin{proof}
Recall that $S=\{0,\ldots,M-2\}$. For $i,j\in S$, set
$$
K_1(i,j):=Q_2(i,j),
\quad
K_n(i,j):=\sum_{x\in S^c}Q_2(i,x)\widetilde{\mathbb P}_x(\widetilde T_S=n-1,\widetilde Y_{n-1}=j)
\quad\text{for }n\ge2.
$$
By \eqref{eq:large.state}, $K_n(i,j)$ is the probability that the homogeneous Markov chain with transition matrix $Q_2$, started at $i$, first returns to $S$ at time $n$ and is then at $j$.

Recall that $U=\{M-1,M,\dots,2M-1\}$ and $\delta_*$ is defined in Lemma~\ref{lem:small_state2}. Applying Lemma~\ref{lem:small_state2} to the homogeneous Markov chain with transition matrix $Q_2$, started at any $i\in S$, we infer that the process makes its first step into $U$ with probability at least $\delta_*$. Until its next visit to $S$, its transition law agrees with $\widetilde Q$ by \eqref{eq:large.state}. By Lemma~\ref{lem:escape}, the probability that it never returns to $S$ is therefore at least
$$
\delta:=\delta_*\min_{x\in U}\widetilde{\P}_x(\widetilde T_S=\infty)>0.
$$
Moreover, since $\widetilde{\mathbb P}_x(\widetilde T_S=n-1,\widetilde Y_{n-1}=j)\le\widetilde{\P}_x(\widetilde Y_{n-1}=j)$, by \eqref{eq:L.sbound}, there exists $C<\infty$ such that
\begin{equation}\label{eq:Q2.return}
\sum_{n=1}^{\infty}\sum_{j\in S}K_n(i,j)\le1-\delta,
\quad
\sum_{j\in S}K_n(i,j)\le C(n+1)^{-D-1}
\quad\text{for }i\in S,\ n\ge1.
\end{equation}
Decomposing according to the first return of the homogeneous Markov chain with transition matrix $Q_2$ to $S$, we have
\begin{equation}\label{eq:Q2.renewal}
Q_2^n(i,j)
=
\sum_{\ell=1}^n\sum_{k\in S}K_\ell(i,k)Q_2^{n-\ell}(k,j)
\quad\text{for }i,j\in S,\ n\ge1.
\end{equation}
Summing \eqref{eq:Q2.renewal} over $1\le n\le N$ and $j\in S$, and including the term at time $0$, we obtain
$$
\max_{i\in S}\sum_{n=0}^N\sum_{j\in S}Q_2^n(i,j)
\le
1+(1-\delta)\max_{i\in S}\sum_{n=0}^N\sum_{j\in S}Q_2^n(i,j).
$$
Letting $N\rightarrow\infty$, we conclude that
\begin{equation}\label{eq:Q2.total}
\max_{i\in S}\sum_{n=0}^{\infty}\sum_{j\in S}Q_2^n(i,j)\le\delta^{-1}.
\end{equation}

Choose $\varepsilon\in(0,1)$ such that
$(1-\delta)(1-\varepsilon)^{-D-1}\le1-{\delta}/{2}.$
We prove by induction that, for a sufficiently large constant $A$,
$$
\sum_{j\in S}Q_2^n(i,j)\le A(n+1)^{-D-1}
\quad\text{for every }i\in S,\ n\ge0.
$$
The case $n=0$ holds for $A\ge1$. Suppose that the bound holds for all smaller times. Splitting the sum over $\ell$ in \eqref{eq:Q2.renewal} according to whether $\ell \le\varepsilon(n+1)$, and using \eqref{eq:Q2.return} and \eqref{eq:Q2.total}, we obtain
\begin{align*}
\sum_{j\in S}Q_2^n(i,j)
&\le
A(1-\delta)(1-\varepsilon)^{-D-1}(n+1)^{-D-1}+
C\varepsilon^{-D-1}(n+1)^{-D-1}
\sum_{m=0}^{n-1}\sum_{k,j\in S}Q_2^m(k,j)\\
&\le
\left(
A\left(1-\frac{\delta}{2}\right)
+\frac{C\varepsilon^{-D-1}(M-1)}{\delta}
\right)(n+1)^{-D-1}.
\end{align*}
Choosing
$A\ge\max\left\{1,\frac{2C\varepsilon^{-D-1}(M-1)}{\delta^2}\right\}$, we complete the induction.
Taking $i=0$, we obtain the result of the lemma.
\end{proof}


\begin{lemma}\label{lem:ruincomp}
For every $d\ge2$ and $i,j\in\mathbb Z_+$, we have
$$
\sum_{h=0}^j Q_d(i+1,h)\le\sum_{h=0}^j Q_d(i,h).
$$
Moreover,
\begin{equation}\label{eq:Qcomp}
\sum_{h=0}^j\widetilde Q(i,h)
\le\sum_{h=0}^jQ_d(i,h)
\le\sum_{h=0}^jQ_2(i,h).
\end{equation}
\end{lemma}
\begin{proof} 
Fix $n\ge 2$ and $1\le k\le n-1$ such that $d=\deg(v_k)$. Since $\Phi_k$ is non-decreasing, for $i_1\le i_2$ we have
$$
\Phi_k(i_1+1)\le\Phi_k(i_2+1)
\quad\text{almost surely}.
$$
By \eqref{def:Y} and the definition of $Q_d$, we obtain the first assertion.

By \eqref{P.quenched}, for $1\le\ell\le M$, the probabilities that the $\ell$-th cookie at $v_k$ prescribes a step to $v_{k+1}$ or $v_{k-1}$ are, respectively,
$$
\frac{1}{d(1+\alpha_\ell)}=\frac{1-\beta_\ell}{d},
\quad
\frac{1+d\alpha_\ell}{d(1+\alpha_\ell)}
=\beta_\ell+\frac{1-\beta_\ell}{d}.
$$
Thus we may represent the $\ell$-th cookie choice by taking a step to $v_{k-1}$ with probability $\beta_\ell$ and, with probability $1-\beta_\ell$, choosing uniformly among the $d$ neighbours of $v_k$.

Let $\mathcal I\subset\{1,\ldots,M\}$ contain each index $\ell$ independently with probability $\beta_\ell$. Conditional on $\mathcal I$, let $\mathcal J\subset\{1,\ldots,M\}\setminus\mathcal I$ contain each remaining index independently with probability $2/d$. Recall that, for $r\ge1$, $H_r$ is the number of failures before the $r$-th success in a sequence of independent $\operatorname{Bernoulli}(1/2)$ random variables defined in the proof of Lemma~\ref{lem:F_pgf}. Here we take this sequence independently of $\mathcal I$ and $\mathcal J$.

Conditional on $\mathcal I$ and $\mathcal J$, we consider the cookie indices in increasing order. Each index in $\mathcal I$ prescribes a backward step from $v_k$ to $v_{k-1}$. At each index in $\mathcal J$, we use the next Bernoulli random variable to prescribe a forward step from $v_k$ to $v_{k+1}$ if its value is $1$ and a backward step from $v_k$ to $v_{k-1}$ if its value is $0$. Indices outside $\mathcal I\cup\mathcal J$ prescribe steps from $v_k$ to neighbours outside $\mathcal P_{v_n}$ and are omitted from the restriction to $\mathcal P_{v_n}$. After index $M$, we continue with the remaining Bernoulli random variables.

The $(i+1)$-st forward step requires $i+1+H_{i+1}$ Bernoulli random variables. For $\ell\in\mathcal I$, exactly $\operatorname{card}\{r\in\mathcal J:r<\ell\}$ Bernoulli random variables are used at cookie indices smaller than $\ell$. Therefore the backward step at index $\ell$ occurs before the $(i+1)$-st forward step if and only if
$$
\operatorname{card}\{r\in\mathcal J:r<\ell\}<i+1+H_{i+1}.
$$
By the definition of $Q_d$, we obtain
\begin{equation}\label{eq:Qcomp.count}
\sum_{h=0}^jQ_d(i,h)
=
\mathbb P\left(
H_{i+1}
+\sum_{\ell\in\mathcal I}
\mathds1_{\left\{
\operatorname{card}\{r\in\mathcal J:r<\ell\}<i+1+H_{i+1}
\right\}}
\le j
\right).
\end{equation}

For the comparison of $Q_2$ and $Q_d$, we use the same set $\mathcal I$ and the same Bernoulli sequence. When $d=2$, we have $\mathcal J=\{1,\ldots,M\}\setminus\mathcal I$, whereas for $d\ge2$, $\mathcal J$ is a subset of this set. For fixed $\mathcal I$ and $H_{i+1}$, each indicator in \eqref{eq:Qcomp.count} is non-increasing when $\mathcal J$ increases under inclusion. Averaging this comparison, we obtain the second inequality in \eqref{eq:Qcomp}.

The sum over $\ell\in\mathcal I$ in \eqref{eq:Qcomp.count} is at most $\operatorname{card}(\mathcal I)$. Hence
$$
\sum_{h=0}^jQ_d(i,h)
\ge
\mathbb P\bigl(H_{i+1}+\operatorname{card}(\mathcal I)\le j\bigr).
$$
By independence and Definition~\ref{def:Ytilde}, for $0\le s<1$,
$$
\mathbb E\left[s^{H_{i+1}+\operatorname{card}(\mathcal I)}\right]
=
(2-s)^{-i-1}\prod_{\ell=1}^M\bigl(1-\beta_\ell(1-s)\bigr)
=
\sum_{h=0}^{\infty}\widetilde Q(i,h)s^h.
$$
Thus $H_{i+1}+\operatorname{card}(\mathcal I)$ has distribution $\widetilde Q(i,\cdot)$, and we obtain the first inequality in \eqref{eq:Qcomp}.
\end{proof}

\begin{proof}[Proof of Proposition \ref{thm:ruin_prob}]
If $M=1$ or $D=0$, the transition laws in \eqref{def:Y} agree with $\widetilde Q$ at every state. In these cases,
$\mathbb P_0(Y_{n-1}^{(n)}=0)=G_{n-1}(0,0),$
and both assertions follow from Lemma~\ref{lem:G.asym}. We therefore assume that $M\ge2$ and $D>0$.

For $n\ge0$ and $i\ge0$, set
$$
P_n(i,0):=\mathbb P_i(Y_n^{(n+1)}=0).
$$
By \eqref{def:Y}, for $n\ge1$, the first transition of $Y^{(n+1)}$ uses $Q_{\deg(v_n)}$. Conditional on $\{Y_1^{(n+1)}=x\}$, the next $n-1$ transitions have the law of $Y^{(n)}$ started at $x$. Therefore $P_0(i,0)=\mathds1_{\{i=0\}}$ and
\begin{equation}\label{eq:ruin.recursion}
P_n(i,0)=\sum_{x=0}^{\infty}Q_{\deg(v_n)}(i,x)P_{n-1}(x,0)
\quad\text{for }n\ge1.
\end{equation}
In particular,
$\mathbb P_0(Y_{n-1}^{(n)}=0)=P_{n-1}(0,0).$
By Lemma~\ref{lem:ruincomp} and induction in \eqref{eq:ruin.recursion}, $P_n(i,0)$ is non-increasing in $i$. Iterating the comparisons in \eqref{eq:Qcomp}, using monotonicity in the initial state for $Q_2^n$, and applying \eqref{eq:Q2bound}, we obtain
\begin{equation}\label{eq:ruin.compare}
\widetilde{\mathbb P}_i(\widetilde Y_n=0)
\le P_n(i,0) \le Q_2^n(i,0) \le Q_2^n(0,0) \le C(n+1)^{-D-1} \quad\text{for }i,n\ge0.
\end{equation}
Taking $i=0$ and replacing $n$ by $n-1$, we have
$$
G_{n-1}(0,0)
\le\mathbb P_0(Y_{n-1}^{(n)}=0)
\le Cn^{-D-1}.
$$
By \eqref{exact.g}, $G_{n-1}(0,0)>0$ for every $n\ge1$. By Lemma~\ref{lem:G.asym},
$$
G_{n-1}(0,0)\sim C(0)n^{-D-1}\quad \text{with }  C(0)>0.
$$
This implies \eqref{eq:ruin.uniform}. 

For $n\ge1$ and $i\ge0$, define
\begin{equation}\label{eq:ruin.Delta}
\Delta_n(i)
:= \sum_{x=0}^{\infty}\bigl(Q_{\deg(v_n)}(i,x)-\widetilde Q(i,x)\bigr)P_{n-1}(x,0)=
P_n(i,0)-\sum_{x=0}^{\infty}\widetilde Q(i,x)P_{n-1}(x,0).
\end{equation}
Since $P_{n-1}(x,0)$ is non-increasing in $x$, by \eqref{eq:Qcomp} we have $\Delta_n(i)\ge0$. By \eqref{eq:large.state}, $\Delta_n(i)=0$ for $i\in S^c$. Moreover, by \eqref{eq:ruin.Delta} and \eqref{eq:ruin.compare},
\begin{equation}\label{eq:ruin.Delta.bound}
0\le\Delta_n(i)\le P_n(i,0)\le C(n+1)^{-D-1}.
\end{equation}
Recall that $L_j^{(i)}(m):=\widetilde{\mathbb P}_i(\widetilde Y_m=j)$, with $L_j^{(i)}(0)=\mathds1_{\{i=j\}}$. Iterating the second equality in \eqref{eq:ruin.Delta}, and using $\Delta_m(j)=0$ for $j\in S^c$, we obtain
\begin{equation}\label{eq:renewal}
P_n(i,0)=L_0^{(i)}(n)+\sum_{m=1}^n\sum_{j\in S}L_j^{(i)}(n-m)\Delta_m(j)
\quad\text{for }i,n\ge0.
\end{equation}

We next estimate $\sum_{m=n}^{\infty}\Delta_m(i)$ for $i\in S$. Fix an integer $R\ge1$. By \eqref{eq:ruin.recursion}, monotonicity of $P_{m-1}(x,0)$ in $x$, and \eqref{eq:Qcomp},
\begin{equation}\label{eq:ruin.decrement}
\begin{aligned}
P_{m-1}(0,0)-P_m(0,0)
&=
\sum_{x=0}^{\infty}Q_{\deg(v_m)}(0,x)
\bigl(P_{m-1}(0,0)-P_{m-1}(x,0)\bigr)\\
&\ge
\left(\sum_{x=R}^{\infty}Q_2(0,x)\right)
\bigl(P_{m-1}(0,0)-P_{m-1}(R,0)\bigr).
\end{aligned}
\end{equation}
For a vertex $v_k$ of degree two, the probability that its first $R$ steps are all from $v_k$ to $v_{k-1}$, before any step from $v_k$ to $v_{k+1}$, is at least $2^{-R}$. By the definition of $Q_2(0,\cdot)$, we therefore have
$\sum_{x=R}^{\infty}Q_2(0,x)\ge2^{-R}>0.$
By \eqref{eq:ruin.Delta}, monotonicity of $P_{m-1}(x,0)$, and \eqref{eq:ruin.decrement}, splitting the sum over $x$ at $R$, we obtain
\begin{equation}\label{eq:ruin.Delta.estimate}
\begin{aligned}
\Delta_m(i)
&\le
\sum_{x=0}^{\infty}\widetilde Q(i,x)
\bigl(P_{m-1}(0,0)-P_{m-1}(x,0)\bigr)\\
&\le
\frac{P_{m-1}(0,0)-P_m(0,0)}
{\sum_{x=R}^{\infty}Q_2(0,x)}
+
P_{m-1}(0,0)\sum_{x>R}\widetilde Q(i,x).
\end{aligned}
\end{equation}
Summing \eqref{eq:ruin.Delta.estimate} over $m\ge n$, using the fact that $P_m(0,0)\rightarrow0$ and \eqref{eq:ruin.compare}, we obtain
$$
n^D\sum_{m=n}^{\infty}\Delta_m(i)
\le
\frac{n^DP_{n-1}(0,0)}
{\sum_{x=R}^{\infty}Q_2(0,x)}
+
C\sum_{x>R}\widetilde Q(i,x),
$$
where the constant $C$ does not depend on $R$. For each fixed $R$, the first term tends to zero as $n\rightarrow\infty$. Hence
$$
\limsup_{n\rightarrow\infty}
n^D\sum_{m=n}^{\infty}\Delta_m(i)
\le C\sum_{x>R}\widetilde Q(i,x).
$$
Taking $R\rightarrow\infty$, we conclude that
\begin{equation}\label{eq:ruin.Delta.tail}
n^D\sum_{m=n}^{\infty}\Delta_m(i)\rightarrow0
\quad\text{as }n\rightarrow\infty,
\quad i\in S.
\end{equation}

For $A(z)=\sum_{m=0}^{\infty}a_mz^m$ with $\sum_{m=0}^{\infty}|a_m|<\infty$, we have
\begin{equation}\label{eq:tailsum}
\sum_{n=1}^{\infty}z^n\sum_{m=n}^{\infty}a_m
=\frac{z\bigl(A(1)-A(z)\bigr)}{1-z}
\quad\text{for }|z|<1.
\end{equation}
Recall that $\widehat L_{0,j}(z)=\sum_{n=1}^{\infty}L_j^{(0)}(n)z^n$, and define
$$
\widehat\Delta_j(z):=\sum_{n=1}^{\infty}\Delta_n(j)z^n,
\quad
\widehat P_{0,0}(z):=\sum_{n=0}^{\infty}P_n(0,0)z^n.
$$
By \eqref{eq:L.sbound}, \eqref{eq:ruin.compare}, and \eqref{eq:ruin.Delta.bound}, these series converge absolutely for $|z|\le1$. Multiplying \eqref{eq:renewal}, with $i=0$, by $z^n$ and summing over $n\ge0$, we obtain
\begin{equation}\label{eq:ruin.generating}
\widehat P_{0,0}(z)
=
1+\widehat L_{0,0}(z)+\widehat\Delta_0(z)
+\sum_{j\in S}\widehat L_{0,j}(z)\widehat\Delta_j(z).
\end{equation}
Here the separate term $\widehat\Delta_0(z)$ corresponds to $L_j^{(0)}(0)=\mathds1_{\{j=0\}}$.

For each $j\in S$, expanding the product in \eqref{eq:ruin.generating} about the values $\widehat L_{0,j}(1)$ and $\widehat\Delta_j(1)$, we have
\begin{equation}\label{eq:ruin.Taylor}
\begin{aligned}
\widehat L_{0,j}(z)\widehat\Delta_j(z) -\widehat L_{0,j}(1)\widehat\Delta_j(1)
&= \widehat\Delta_j(1) \bigl(\widehat L_{0,j}(z)-\widehat L_{0,j}(1)\bigr) + \widehat L_{0,j}(1) \bigl(\widehat\Delta_j(z)-\widehat\Delta_j(1)\bigr)\\
&\quad+ \bigl(\widehat L_{0,j}(z)-\widehat L_{0,j}(1)\bigr) \bigl(\widehat\Delta_j(z)-\widehat\Delta_j(1)\bigr).
\end{aligned}
\end{equation}
By \eqref{eq:L.sbound} and \eqref{eq:ruin.Delta.bound}, both $L_j^{(0)}(m)$ and $\Delta_m(j)$ are bounded by $C(m+1)^{-D-1}$. Splitting the following sum at $m=\lfloor n/2\rfloor$, we obtain, for $n\ge2$,
$$
\sum_{m=1}^{n-1}L_j^{(0)}(m)\sum_{t=n-m}^{n-1}\Delta_t(j)
\le Cn^{-D-1}\sum_{m=1}^{n}m^{-D}.
$$
Using \eqref{eq:tailsum} and these bounds, we have
\begin{equation}\label{eq:HL.rembound}
\begin{aligned}
&n^D\left|
[z^n]\frac{z}{1-z}
\bigl(\widehat L_{0,j}(z)-\widehat L_{0,j}(1)\bigr)
\bigl(\widehat\Delta_j(z)-\widehat\Delta_j(1)\bigr)
\right|\\
&\quad=n^D\left|
\sum_{m=1}^{n-1}L_j^{(0)}(m)\sum_{t=n-m}^{n-1}\Delta_t(j)
-\left(\sum_{m=n}^{\infty}L_j^{(0)}(m)\right)
\left(\sum_{t=n}^{\infty}\Delta_t(j)\right)
\right|\\
&\quad\le
C\left(\frac1n\sum_{m=1}^n m^{-D}+n^{-D}\right)
\rightarrow0
\quad\text{as }n\rightarrow\infty,
\end{aligned}
\end{equation}
where $[z^n]$ denotes the coefficient of $z^n$ and we used $D>0$ in the last limit.
Subtract \eqref{eq:ruin.generating} at $z=1$ from its value at $z$ and multiply by $-z/(1-z)$. Comparing the coefficients of $z^n$, using \eqref{eq:tailsum} and \eqref{eq:ruin.Taylor}, and then applying \eqref{eq:L.taillim}, \eqref{eq:ruin.Delta.tail}, and \eqref{eq:HL.rembound}, we obtain
\begin{equation}\label{eq:ruin.tail}
n^D\sum_{m=n}^{\infty}P_m(0,0)
\rightarrow
\frac1D\left(\ell_0+\sum_{j\in S}\ell_j\widehat\Delta_j(1)\right)
\quad\text{as }n\rightarrow\infty.
\end{equation}
Set
$$
c_*(\varkappa):=\ell_0+\sum_{j\in S}\ell_j\widehat\Delta_j(1).
$$
By \eqref{eq:Lasym}, $\ell_j\ge0$ and $\ell_0>0$. Since $\Delta_m(j)\ge0$ and $\widehat\Delta_j(1)<\infty$, we have $c_*(\varkappa)\in(0,\infty)$.

By Lemma~\ref{lem:monotone}, $P_n(0,0)$ is non-increasing. Applying Theorem~\ref{thm:karamata} to \eqref{eq:ruin.tail}, we obtain
$$
P_n(0,0)\sim c_*(\varkappa)n^{-D-1}
\quad\text{as }n\rightarrow\infty.
$$
Since $\mathbb P_0(Y_{n-1}^{(n)}=0)=P_{n-1}(0,0)$, this proves the first assertion. Applying \eqref{eq:ruin.uniform} along the ray $\varkappa$ and using this asymptotic, we obtain $C^{-1}\le c_*(\varkappa)\le C$.
\end{proof}

\subsection{Generalised ruin probability}\label{sec:gen_ruin}
Let $\tau_\rho^{(j)}(v)$ be the $j$-th return time of $\X^{(v)}$ to $\rho$. We estimate the generalised ruin probability
$$
r_v^{(j)}:=\P\bigl(\tau_v(v)<\tau_\rho^{(j)}(v)\bigr).
$$

\begin{proposition}\label{lem:genruin}
There exists a constant $C<\infty$, depending only on $M$ and $(\alpha_k)_{1\le k\le M}$, such that, for every $v\in V\setminus\{\rho\}$ and every integer $j\ge1$,
$$
r_v^{(j)} \le C\left(1\wedge\left(\frac{j}{|v|+1}\right)^{D+1}\right).
$$
\end{proposition}
\begin{proof}
Let $|v|=n$ and $\mathcal P_v=(v_0,\ldots,v_n)$, with $v_0=\rho$ and $v_n=v$. Notice that
$\{\tau_v(v)<\tau_\rho^{(j)}(v)\} =\{B_n(1)<j\}.$
By \eqref{eq:comp}, we therefore obtain
\begin{equation}\label{genruin.bw}
r_v^{(j)} =\P_0\bigl(Y_{n-1}^{(n)}<j\bigr).
\end{equation}

We first estimate the corresponding probability for $\widetilde Y$. For $j\ge1$, set $s=j/(j+1)$. Since $s^{-j}\le\mathrm e$, by Lemma~\ref{lem:G.asym}, for every $x\in\Z_+$ and $m\ge0$, we have
\begin{equation}\label{genruin.hom}
\begin{aligned}
\widetilde{\P}_x(\widetilde Y_m<j) \le s^{-j}G_m(x,s)\le \mathrm e\left(1+\frac{m}{j+1}\right)^{-D-1}\le C\left(1\wedge\left(\frac{j}{m+1}\right)^{D+1}\right).
\end{aligned}
\end{equation}
If $M=1$ or $D=0$, the transition laws of $Y^{(n)}$ and $\widetilde Y$ agree at every state. In these cases, applying \eqref{genruin.hom} with $m=n-1$ and using $n+1\le2n$, we obtain the result from \eqref{genruin.bw}. We now assume that $M\ge2$ and $D>0$.

For each fixed $n$, recall that $T_S^+:=\inf\{1\le\ell\le n-1:Y_\ell^{(n)}\in S\}$ and set
$$
\quad
K_\ell(i,m):=\P_i(T_S^+=\ell,Y_\ell^{(n)}=m)
\quad\text{for } i,m\in S.
$$
For $\ell \ge2$, on $\{T_S^+=\ell\}$ we have $Y_1^{(n)}\in S^c$, and the transitions until the return to $S$ agree with those of $\widetilde Y$. Hence
$$
\sum_{k\in S}K_\ell(i,k) \le \sup_{x\in S^c} \widetilde{\P}_x(\widetilde Y_{\ell-1}\in S).
$$
Applying \eqref{genruin.hom} with $j=M-1$, we obtain a bound by $C(\ell+1)^{-D-1}$. Moreover, by Lemmas~\ref{lem:small_state2} and~\ref{lem:escape}, we have
$$
\P_i(T_S^+>n-1) \ge \delta_*\min_{x\in U} \widetilde{\P}_x(\widetilde T_S=\infty) =:\delta>0.
$$
Consequently, there exists $C_0\ge1$, independent of the finite path, such that
\begin{equation}\label{eq:K}
\begin{aligned}
\sum_{\ell=1}^{n-1}\sum_{k\in S}K_\ell(i,k)&\le1-\delta,\quad
\sum_{k\in S}K_\ell(i,k)\le C_0(\ell+1)^{-D-1}\quad\text{for } 1\le \ell \le n-1.
\end{aligned}
\end{equation}

By the strong Markov property and \eqref{def:Y}, conditionally on $\{T_S^+=\ell,Y_\ell^{(n)}=k\}$, the chain after time $\ell$ has the transition law of $Y^{(n-\ell)}$ started at $k$, associated with the path  $\mathcal{P}_{v_{n-\ell}}$. Thus, for $i\in S$ and $0\le m\le n-1$,
\begin{equation}\label{genruin.recur}
\begin{aligned}
\P_i(Y_m^{(n)}<j) &= \P_i(Y_m^{(n)}<j,T_S^+>m)+ \sum_{\ell=1}^m\sum_{k\in S} K_\ell(i,k)\P_k(Y_{m-\ell}^{(n-\ell)}<j)\\
&\le C_0j^{D+1}(m+1)^{-D-1}+ \sum_{\ell=1}^m\sum_{k\in S} K_\ell(i,k)\P_k(Y_{m-\ell}^{(n-\ell)}<j).
\end{aligned}
\end{equation}
For the last inequality, when $m\ge1$, the event $\{T_S^+>m\}$ requires $Y_1^{(n)}\in S^c$. We then apply \eqref{genruin.hom} to the subsequent $m-1$ transitions.
Choose $\varepsilon\in(0,1)$ such that
$(1-\delta)(1-\varepsilon)^{-D-1}\le1-{\delta}/{2}$
and an integer $L\ge2$ such that
$C_0\varepsilon^{-D-1}\sum_{\ell=0}^{\infty}(\ell+L)^{-D-1}\le\frac{\delta}{4}.$
Splitting the sum over $\ell$ according to whether $\ell\le\varepsilon(m+L)$ and using \eqref{eq:K}, we obtain
\begin{equation}\label{eq:genruin.weighted}
\begin{aligned}
\sum_{\ell=1}^m\sum_{k\in S}
K_\ell(i,k)(m-\ell+L)^{-D-1} &\le (m+L)^{-D-1} \Big((1-\delta)(1-\varepsilon)^{-D-1} +C_0\varepsilon^{-D-1} \sum_{\ell=0}^{\infty}(\ell+L)^{-D-1} \Big)\\
&\le
\left(1-\frac{\delta}{4}\right)(m+L)^{-D-1}.
\end{aligned}
\end{equation}
Choose $C\ge4C_0L^{D+1}/\delta$. We prove by induction on $m$, simultaneously for all finite paths, all $n\ge m+1$, $i\in S$ and $j\ge1$, that
$$
\P_i(Y_m^{(n)}<j)
\le Cj^{D+1}(m+L)^{-D-1}.
$$
This inequality trivally hold for $m=0$. Suppose that it holds for all smaller times. By \eqref{genruin.recur}, the induction hypothesis and \eqref{eq:genruin.weighted}, we obtain
$$
\begin{aligned}
\P_i(Y_m^{(n)}<j) &\le C_0j^{D+1}(m+1)^{-D-1}+ Cj^{D+1}\sum_{\ell=1}^m\sum_{k\in S} K_\ell(i,k)(m-\ell+L)^{-D-1}\\
&\le \left( C_0L^{D+1} +C\left(1-\frac{\delta}{4}\right) \right) j^{D+1}(m+L)^{-D-1}\le Cj^{D+1}(m+L)^{-D-1}.
\end{aligned}
$$
This completes the induction. Taking $i=0$ and $m=n-1$, by \eqref{genruin.bw} and $L\ge2$, we conclude that
$$
r_v^{(j)} \le Cj^{D+1}(n-1+L)^{-D-1} \le C\left(\frac{j}{n+1}\right)^{D+1}.
$$
Combining this estimate with $r_v^{(j)}\le1$, we obtain the stated bound.
\end{proof}

\section{Markov chains related to the number of forward steps}\label{sec:forward}

In this section we study $Z^{(n)}$, the inhomogeneous Markov chain describing the numbers of forward steps made by the extension $\X^{(v_n)}$ on $\mathcal P_{v_n}$ before its first return to $\rho$, where $v_n$ is at distance $n$ from $\rho$. We provide moment bounds for $Z^{(n)}$, uniform over all finite paths, which will be used in the proof of Proposition~\ref{prop:per} to establish quasi-independence of the ruin percolation.

\subsection{Markovian structure of the forwards steps.}
As in the previous section, we fix an infinite self-avoiding path $\varkappa=(v_k)_{k\ge0}$, with $v_0=\rho$.

Fix $n\ge1$. We consider the extension process $\X^{(v_n)}$ on the path $\mathcal P_{v_n}$. For $0\le k\le n-1$, let $F_n(k)$ be the number of forward steps from $v_k$ to $v_{k+1}$ before the first return of $\X^{(v_n)}$ to $\rho$. Note that $F_n(0)=1$.

Recall the coupling of $\boldsymbol\xi$ with the Bernoulli random variables $\boldsymbol\eta=(\eta_x(j):x\in\varkappa,\ j\ge1)$ from Section~\ref{sec:markov}. For integers $z\ge0$ and $1\le k\le n-1$, let
$$
\Psi_k(z):=\inf\left\{\ell\in\mathbb Z_+:\sum_{j=1}^{\ell+z}(1-\eta_{v_k}(j))=z\right\}.
$$
For $z\ge1$, $\Psi_k(z)$ is the number of forward steps from $v_k$ to $v_{k+1}$ before the $z$-th backward step from $v_k$ to $v_{k-1}$. Also, $\Psi_k(0)=0$.

Before the first return to $\rho$, the number of steps from $v_k$ to $v_{k-1}$ equals $F_n(k-1)$. If $v_k$ is visited, the last step from $v_k$ before that return is to $v_{k-1}$. If $v_k$ is not visited, then $F_n(k-1)=F_n(k)=0$. Therefore
$$
F_n(k)=\Psi_k(F_n(k-1))
\quad\text{for }1\le k\le n-1.
$$
Each $\Psi_k$ is determined by $(\eta_{v_k}(j))_{j\ge1}$, and these sequences are independent for distinct $k$ by the construction in Section~\ref{sec:markov}. Hence
$$
(F_n(0),F_n(1),\ldots,F_n(n-1))
$$
is an inhomogeneous Markov chain, whose transition law may depend on $k$ through $\deg(v_k)$.

For each $n\ge1$, let $(Z_m^{(n)})_{0\le m\le n-1}$ be the Markov chain on $\mathbb Z_+$ with transition law
$$
\P_i\bigl(Z_{m+1}^{(n)}=j\mid Z_m^{(n)}=\ell\bigr)=\P\bigl(\Psi_{m+1}(\ell)=j\bigr)\quad\text{for }0\le m\le n-2,\ \ell,j\in\mathbb Z_+.
$$
Here $\P_i$ denotes its law with initial state $i$. Under $\P_1$, we have
\begin{equation}\label{eq:forward.comp}
(F_n(0),F_n(1),\ldots,F_n(n-1))\overset{\mathrm d}{=}(Z_0^{(n)},Z_1^{(n)},\ldots,Z_{n-1}^{(n)}).
\end{equation}

\subsection{Moment bounds for the forward chain.}

Recall from the proof of Lemma \ref{lem:F_pgf} that $A_u^i$ (resp.\ $B_u^i$) be the indicator that, on the $i$th visit to $u$, the walk makes a forward step along $\mathcal P_v$ (resp.\ a backward step along $\mathcal P_v$), so that $A_u = \sum_{i=1}^M A_u^i$ and $B_u = \sum_{i=1}^M B_u^i$. Notice that it is possible that $A_u^i=B_u^i=0$, which corresponds to the event that the $i$th departure from $u$ is to a child not lying in $\mathcal P_v$. Notice that we have decomposition, $$\Psi_x(z) = A_x + H_{z-B_x},$$
where $H_r$, with $r\ge 1$, is the number of backward steps before the $r$-th forward step after the cookies have been consumed, which is negative binomial distributed, counting the number of failures before the $r$-th success in a sequence of independent $\operatorname{Bernoulli}(1/2)$ trials. We set $H_0=0$.

\begin{proposition}\label{prop.moment-drift}
There exists a constant $C_0<\infty$, depending only on $M$ and $(\alpha_j)_{1\le j\le M}$, such that
$$
\E_1\bigl[(Z_k^{(n)})^{2(D+1)}\bigr]\le C_0(k+1)^{D+1}
\quad\text{for }n\ge1,\ 0\le k\le n-1,
$$
uniformly over all rays $\varkappa$.
\end{proposition}

\begin{proof} For $n=1$, we have $Z_0^{(1)}=1$. Fix $n\ge2$. 
We first prove that there exists a constant $C_1\in(0,\infty)$ such that, for $0\le k\le n-2$ and $z\in\mathbb Z_+$,
\begin{align}\label{eq:increment}
\E_1\bigl[(Z_{k+1}^{(n)})^{2(D+1)}-(Z_k^{(n)})^{2(D+1)}
\mid Z_k^{(n)}=z\bigr]
\le C_1z^{2D+1}.
\end{align}
For $0\le z\le M$, by monotonicity of $\Psi_{k+1}$, we have $\Psi_{k+1}(z)\le\Psi_{k+1}(M).$
Notice also that
$$
\Psi_{k+1}(M)
=A_{k+1}+H_{M-B_{k+1}}
\le_{\mathrm{st}}M+H_M.
$$
Hence, for $1\le z\le M$,
\begin{equation}\label{eq:Zbound1}
\mathbb E_1\!\left[(Z_{k+1}^{(n)})^{2(D+1)}
\mid Z_k^{(n)}=z\right]
\le\mathbb E[(M+H_M)^{2(D+1)}]
=:C_2<\infty.
\end{equation}

Since $\mathbb E[A_{k+1}-B_{k+1}]=-D$, for $z>M$, we have
$$
\begin{aligned}
\mathbb E[\Psi_{k+1}(z)]&=z-D,\\
\operatorname{Var}(\Psi_{k+1}(z))
&=2\bigl(z-\mathbb E[B_{k+1}]\bigr)
+\operatorname{Var}(A_{k+1}-B_{k+1})
\le C_3z.
\end{aligned}
$$
For $r\ge1$, we write $H_r$ as a sum of $r$ independent non-negative geometric random variables with mean $1$. Using the fact that $D\le M$ and H\"older's inequality, we obtain
\begin{equation}\label{eq:Gmom}
\mathbb E[|H_r-r|^{2D+2}]
\le
\mathbb E[|H_r-r|^{2M+2}]^{\frac{D+1}{M+1}}
\le C_4r^{D+1}.
\end{equation}
By \eqref{eq:Gmom}, we thus have
$$
\begin{aligned}
\mathbb E\!\left[
\left|\Psi_{k+1}(z)-(z-D)\right|^{2D+2}
\right]
&\le
2^{2D+1}\mathbb E\!\left[
\left|H_{z-B_{k+1}}-(z-B_{k+1})\right|^{2D+2}
\right]\\
&\quad+
2^{2D+1}\mathbb E\!\left[
|A_{k+1}-B_{k+1}+D|^{2D+2}
\right]\\
&\le C_5z^{D+1}.
\end{aligned}
$$
Using the Taylor expansion with the integral form of the remainder term, for $z>M$ and $t\ge0$, we have
\begin{equation}
\begin{aligned}
0
&\le t^{2D+2}-(z-D)^{2D+2}
-(2D+2)(z-D)^{2D+1}(t-z+D)\\
&=(2D+2)(2D+1)(t-z+D)^2
\int_0^1(1-u)\bigl(z-D+u(t-z+D)\bigr)^{2D}\rmd u\\
&\le C_6\left((z-D)^{2D}(t-z+D)^2
+|t-z+D|^{2D+2}\right).
\end{aligned}
\end{equation}
Taking $t=\Psi_{k+1}(z)$ and using $\mathbb E[\Psi_{k+1}(z)]=z-D$, we obtain
$$
\begin{aligned}
\mathbb E[\Psi_{k+1}(z)^{2D+2}]-z^{2D+2}
&\le (z-D)^{2D+2}-z^{2D+2}
+C_6\left(C_3z^{2D+1}+C_5z^{D+1}\right)\\
&\le C_6(C_3+C_5)z^{2D+1}.
\end{aligned}
$$
Also note that $\Psi_{k+1}(0)=0$. Choosing $C_1=\max\{C_2,C_6(C_3+C_5)\},$
we conclude that \eqref{eq:increment} holds true for every $z\in\mathbb Z_+$.

By \eqref{eq:increment} and $Z_0^{(n)}=1$, all the moments below are finite by induction. Set
$$
x_k=\E_1[(Z_k^{(n)})^{2(D+1)}].
$$
Taking expectations in \eqref{eq:increment}, we have
$$
x_{k+1}\le x_k+
C_1\E_1\!\left[
(Z_k^{(n)})^{2D+1}\mathds1_{\{Z_k^{(n)}\ge1\}}
\right].
$$
Using H\"older's inequality with exponents $\frac{2D+2}{2D+1}$ and $2D+2$, we obtain
\begin{align*}
x_{k+1}
&\le x_k+
C_1\E_1[(Z_k^{(n)})^{2D+2}]^{\frac{2D+1}{2D+2}}
\E_1[(\mathds1_{\{Z_k^{(n)}\ge1\}})^{2D+2}]^{\frac1{2D+2}}\\
&=x_k+
C_1x_k^{\frac{2D+1}{2D+2}}
\P_1(Z_k^{(n)}\ge1)^{\frac1{2D+2}}.
\end{align*}
Recall that $F_n(k)$ counts the steps from $v_k$ to $v_{k+1}$ made by $\X^{(v_n)}$ before its first return to $\rho$. Thus $\{F_n(k)\ge1\}$ is the event that $\X^{(v_n)}$ reaches $v_{k+1}$ before returning to $\rho$. By the restriction principle, $\X^{(v_n)}$ and $\X^{(v_{k+1})}$ agree until their first visit to $v_{k+1}$ or their first return to $\rho$. Using \eqref{eq:forward.comp}, we therefore obtain
$$
\P_1(Z_k^{(n)}\ge1)=\P(F_n(k)\ge1)=r_{k+1},
$$
Applying Proposition~\ref{thm:ruin_prob}, we thus have
$$
\P_1(Z_k^{(n)}\ge1) \le C_7(k+1)^{-D-1}.
$$
Consequently,
\begin{equation}\label{eqn:reccurence}
x_{k+1}
\le x_k+
C_8x_k^{\frac{2D+1}{2D+2}}(k+1)^{-1/2},
\end{equation}
where $C_8=C_1C_7^{1/(2D+2)}$. Choosing
$$
C_0\ge
\max\left\{1,\left(\frac{C_8}{D+1}\right)^{2(D+1)}\right\},
$$
we complete the proof by showing that
$$x_k\le C_0(k+1)^{D+1}.$$
We proceed by induction. Since $Z_0^{(n)}=1$, we have
$x_0=1\le C_0$. Assume that the assertion holds true up to $k$. By \eqref{eqn:reccurence}, we have
\begin{align*}
x_{k+1}
&\le C_0(k+1)^{D+1}
+C_8\bigl(C_0(k+1)^{D+1}\bigr)^{\frac{2D+1}{2D+2}}
(k+1)^{-1/2}
\\
&= C_0(k+1)^{D+1}
+C_8C_0^{\frac{2D+1}{2D+2}}(k+1)^D\le C_0(k+2)^{D+1},
\end{align*}
where in the last inequality, we use the fact that $(D+1)(k+1)^D
\le(k+2)^{D+1}-(k+1)^{D+1}$. The proof completes by the induction principle.
   
\end{proof}




\section{Quasi-independent percolation} \label{sec:perc}
In this section we prove Theorem~\ref{main.thm}. We first recall percolation criteria from \cite{Lyons1989,LP2016} and then define the ruin percolation for $\X$, a bond percolation on $\mathcal T$ in which the edge $\{v^{-1},v\}$ is open if the extension process $\X^{(v)}$ visits $v$ before its first return to $\rho$. Using Propositions~\ref{prop.res}, \ref{thm:ruin_prob}, \ref{lem:genruin} and~\ref{prop.moment-drift}, we show that the ruin percolation is quasi-independent. Applying the percolation criteria together with the ruin probability estimates, we show that the event that the open cluster containing $\rho$ is infinite has probability $0$ when ${\rm br}_r(\mathcal T)<D+1$ and positive probability when ${\rm br}_r(\mathcal T)>D+1$. Finally, combining these conclusions with the connection between the ruin percolation and the recurrence or transience of $\X$, and with a $0-1$ law that we establish in this section, we complete the proof.

\subsection{Preliminary notations and results}\label{sec:quasip}
Let $\Tcal=(V,E)$ be a infinite locally finite tree. For each edge $e\in E$, we assign a Bernoulli random variable $\xi_e$ with parameter $p_e\in [0,1]$. The Bernoulli field $(\xi_e)_{e\in E}$ is not necessarily independent nor identically distributed. We say that $e$ is \textbf{open} if $\xi_e=1$, and \textbf{closed} otherwise. Assume that $(\xi_e)_{e\in E}$ is governed by a probability measure $\mathbf{Q}$. We call $\mathbf Q$ a \textbf{bond percolation} on $\Tcal$. 
After removing all closed edges from $E$, we obtain connected components comprising of open edges, which we call \textbf{clusters}. If two vertices $x$ and $y$ belong to the same cluster, we write $x\leftrightarrow y$. If the cluster containing $x$ has infinitely many vertices, we write $x\leftrightarrow \infty$. For two vertices $x,y$, we denote by $x\wedge y$ their nearest common ancestor.

A bond percolation $\mathbf{Q}$ is \textbf{quasi-independent} if there exists a constant $M>0$ where 
$$\mathbf{Q}(\rho\leftrightarrow x, \rho\leftrightarrow y \ |\ \rho\leftrightarrow x\wedge y )\le M\cdot \mathbf{Q}(\rho\leftrightarrow x \ |\ \rho\leftrightarrow x\wedge y )\mathbf{Q}(\rho\leftrightarrow y \ |\ \rho\leftrightarrow x\wedge y )$$
for each $x, y\in V$.

For each edge $e=\{e^-,e^+\} \in E$ with $|e^+|=|e^-|+1$, let 
\begin{align}\label{adt.c} \text{ $c(e) = 1$ for $|e| = 1$}\quad \text{and}\quad
    c(e) = \frac{\mathbf{Q}(\rho\leftrightarrow e^+)}{\mathbf{Q}(e \text{ is closed} \mid \rho\leftrightarrow e^-)}.
\end{align}
for $|e|>1$. We call $(c(e))_{e\in E}$ the \textbf{adapted conductances} of the percolation $\mathbf{Q}$. We will use the following result:

\begin{proposition}[Theorem 5.19 in \cite{LP2016}] \label{lyons1989}
    Let $\mathbf{Q}$ be a quasi-independent percolation process taking place on an infinite locally finite tree $\Tcal = (V,E)$.
    \begin{itemize}
        \item [(i)] If \(\inf_{\pi \in \Pi} \sum_{e \in \pi} \mathbf{Q}(\rho \leftrightarrow e) = 0 \text{ then } \mathbf{Q}(\rho \leftrightarrow \infty) = 0\);
        \item [(ii)] If there exists a non-zero flow $\theta$ such that $\sum_{e\in E} \frac{\theta(e)^2}{c(e)}<\infty$ then  $\mathbf{Q}(\rho \leftrightarrow \infty) > 0.$
    \end{itemize}
\end{proposition}
\subsection{Ruin percolation}\label{sec:ruinper}
Let $\tau_v=\inf\{n\ge0: X_n=v\}$ and $\tau_v^+=\inf\{n>\tau_v: X_n=v\}$ be respectively the first hitting time of vertex $v$ and the first return time to vertex $v$. 
Define
$$
\mathcal{C} ( \rho)=\left\{\{v^{-1},v\} \in E \colon \tau_v  < \tau_{\rho}^+ \right\}. 
$$ 

Let $\tau_u(v)$ and $\tau_{u}^{+}(v)$ be respectively the first hitting times and the return time to vertex $u$ associated with $\X^{(v)}$.
Let 
$$\mathcal{C}_{\rm CP} ( \rho)=\left\{\{v^{-1},v\} \in E \colon \tau_v(v)  < \tau_{\rho}^+(v) \right\}. $$
We say an edge $e \in E$ is \textbf{open} if $e \in{\mathcal{C}}_{\mathrm{CP}} ( \rho)$, and \textbf{closed} otherwise. We define a correlated percolation by removing all closed edges, and we refer to this model as the \textbf{ruin percolation}.

The following result follows by the same proof as in \cite[Lemma 3.3]{LN2026} and \cite[Lemma 7.1]{CKS2020}.

\begin{lemma}[Lemma 3.3 in \cite{LN2026}, Lemma 7.1 in \cite{CKS2020}] \label{per2.lem} We have
\begin{align*}
    \P(\tau^+_{\rho} = \infty) = \P (|\mathcal{C}(\rho)| = \infty) = \P(|\mathcal{C}_{\rm CP}(\rho)| = \infty).
\end{align*} 
\end{lemma}


In this section we aim to prove the following:

\begin{proposition}\label{prop:per}
        The ruin percolation is quasi-independent. 
\end{proposition}

Consider $e_1, e_2\in E$ with $e=e_1\wedge e_2$. 
Let $\widetilde \X$ be the the extension of $\X$ on $\widetilde{\mathcal{T}}:=\mathcal{P}_{e_1^+}\cup \mathcal{P}_{e_2^+}$. 
For $e=\{e^-, e^+\}\in E$ with $|e^+|=|e^-|+1$, let
$$N_e:=\sum_{k=0}^{\widetilde\tau_\rho^+-1}
\mathds1_{\{\widetilde X_k=e^+,\,\widetilde X_{k+1}=e^-\}}.
$$
Then, by the law of total probability,
\begin{align}
&\P\bigl(e_1,e_2\in \mathcal C_{\rm CP}(\rho)\mid e\in \mathcal C_{\rm CP}(\rho)\bigr)  =\sum_{j\ge1}\P(N_e=j\mid N_e\ge 1) \P\bigl(e_1,e_2\in \mathcal C_{\rm CP}(\rho)\mid N_e=j\bigr).
\label{eq:branch-mixture}
\end{align}

We have  
\begin{align}\label{eq:momentZ}
\E[ N_{e}^q \mid N_{e}\ge 1]=\frac{\sum_{j=1}^{\infty} j^q \P( N_e =j) }{\P( N_{e}\ge 1)},
\end{align}
where we note that \begin{itemize}
    \item $\P( N_{e}\ge 1)$ is the ruin probability of the extension process on $\mathcal{P}_e$;
    \item $\sum_{j=1}^{\infty} j^q \P( N_e =j)$ is related to the $q$-th moment of the Markov chain $Z_n$.
\end{itemize}
Hence, to prove that the ruin percolation is quasi-independent, it is sufficient to prove that
\begin{lemma}\label{lem:branch_corr}
Let $e_1,e_2\in E$ with $e=e_1\wedge e_2$, and assume that $e_1$ and $e_2$ lie strictly below $e$ in two distinct descendant subtrees. Then there exists a constant $C<\infty$, depending only on $M$ and $(\alpha_k)_{1\le k\le M}$, such that for every integer $j\ge1$ with $\P(N_e=j)>0$, we have
$$
\P(e_1,e_2\in\mathcal C_{\rm CP}(\rho)\mid N_e=j) \le C \left(1\wedge\left(\frac{j}{|e_1|-|e|+1}\right)^{D+1}\right) \left(1\wedge\left(\frac{j}{|e_2|-|e|+1}\right)^{D+1}\right).
$$
\end{lemma}

\begin{proof}
For $i\in\{1,2\}$, let $[e^+,e_i^+]$ denote the path from $e^+$ to $e_i^+$, including both endpoints, and set $m_i:=|e_i|-|e|+1$.

Let $\X^{[e^+,e_i^+]}$ be the extension process on $[e^+,e_i^+]$, constructed similarly as for extensions on $\mathcal{P}_v$ in Definition~\ref{def:restricted} using the same random variables $\boldsymbol\xi$, such that it starts from $e^+$. The same restriction principle as in Proposition~\ref{prop.res} applies to $[e^+,e_i^+]$. Hence the restriction of $\widetilde\X$ to $[e^+,e_i^+]$ agrees with $\X^{[e^+,e_i^+]}$ for as long as the restriction is defined.

For $k\ge1$, let $r_{[e^+,e_i^+]}^{(k)}$ denote the probability that $\X^{[e^+,e_i^+]}$ reaches $e_i^+$ before its $k$-th return to $e^+$ after time $0$, and set $r_{[e^+,e_i^+]}^{(0)}:=0.$ Let $L_i$ be the number of steps of $\widetilde{\mathbf{X}}$ from the common vertex $e^+$ into the branch containing $e_i$ before its first return to $\rho$. 

By the restriction principle applied to $\mathcal P_{e^+}$, $N_e$ agrees with the number of steps from $e^+$ to $e^-$ of $\X^{(e^+)}$ before its first return to $\rho$. By the construction, the extensions $\X^{[e^+,e_1^+]}$ and $\X^{[e^+,e_2^+]}$ are independent of each other and of $(N_e,L_1,L_2)$.

If $L_i=0$, we have $e_i\notin\mathcal C_{\rm CP}(\rho)$. On $\{N_e=j,L_i\ge1\}$, the restriction of $(\widetilde X_k)_{0\le k<\widetilde\tau_\rho^+}$ to $[e^+,e_i^+]$ ends at its $L_i$-th return to $e^+$. By Proposition~\ref{prop.res}, the event $\{e_i\in\mathcal C_{\rm CP}(\rho)\}$ therefore occurs precisely when $\X^{[e^+,e_i^+]}$ reaches $e_i^+$ before that return.  We therefore have
\begin{equation}\label{eq:ruin-bound}
\P(e_1,e_2\in\mathcal C_{\rm CP}(\rho)\mid N_e=j)
=
\E\left[
r_{[e^+,e_1^+]}^{(L_1)} r_{[e^+,e_2^+]}^{(L_2)}
\mid N_e=j
\right].
\end{equation}
The same result of Proposition~\ref{lem:genruin} applies to $\X^{[e^+,e_i^+]}$. We therefore have
\begin{equation}\label{eq:genruin}
r_{[e^+,e_i^+]}^{(L_i)} \le C\left(1\wedge\left(\frac{L_i}{m_i}\right)^{D+1}\right) \quad\text{for } i\in\{1,2\},
\end{equation}
We also have
\begin{equation}\label{ineq1}
\prod_{i=1}^2 \left(1\wedge\left(\frac{L_i}{m_i}\right)^{D+1}\right)
\le \prod_{i=1}^2 \left(1\wedge\left(\frac{j}{m_i}\right)^{D+1}\right) \left(1+\frac{L_1+L_2}{j}\right)^{2D+2}.
\end{equation}
Conditional on $\{N_e=j\}$, $L_1+L_2$ is stochastically bounded by $M$ plus the sum of $j$ independent geometric random variables, each with probability $\frac13(\frac23)^\ell$ of taking the value $\ell\ge0$. Using $M/j\le M$ and Jensen's inequality, we thus have
\begin{equation}\label{ineq2}
\begin{aligned}
\E\left[
\left(1+\frac{L_1+L_2}{j}\right)^{2D+2}
\,\middle|\,N_e=j
\right] &\le
\sum_{\ell_1,\ldots,\ell_j\ge0}
\left(\prod_{k=1}^j\frac13\left(\frac23\right)^{\ell_k}\right)
\left(1+M+\frac{\ell_1+\cdots+\ell_j}{j}\right)^{2D+2}\\
&\quad\le
\sum_{\ell=0}^{\infty}
\frac13\left(\frac23\right)^\ell
(1+M+\ell)^{2D+2}
\le C,
\end{aligned}
\end{equation}
where $C$ is a constant depends on $M$ and $D$. 
Combining \eqref{eq:ruin-bound} with \eqref{eq:genruin}, \eqref{ineq1} and \eqref{ineq2}, we obtain the result of the lemma.

\end{proof}

\begin{proof}[Proof of Proposition \ref{prop:per}]

By Proposition~\ref{prop.res}, every ancestor of an edge in $\mathcal C_{\rm CP}(\rho)$ also belongs to $\mathcal C_{\rm CP}(\rho)$. If one of $e_1,e_2$ is an ancestor of the other, the defining inequality for quasi-independence is an equality. If $\mathcal P_{e_1^+}\cap\mathcal P_{e_2^+}=\{\rho\}$, then, by the construction in Section~\ref{sec:ext}, for $i\in\{1,2\}$,
$$
\{e_i\in\mathcal C_{\rm CP}(\rho)\}
\in
\sigma\bigl(\xi(u,w,k):u\in\mathcal P_{e_i^+}\setminus\{\rho\},\ w\sim u,\ k\ge1\bigr).
$$
By independence of $\boldsymbol\xi$, the two events $\{e_1\in\mathcal C_{\rm CP}(\rho)\}$ and $\{e_2\in\mathcal C_{\rm CP}(\rho)\}$ are independent.

It remains to consider the case in which $e=e_1\wedge e_2\in E$ and $e_1,e_2$ lie strictly below $e$ in two distinct descendant subtrees. Set
$$
n:=|e|,\quad m_i:=|e_i|-|e|+1
\quad\text{for }i\in\{1,2\}.
$$
By Proposition~\ref{thm:ruin_prob}, there exists $C\in[1,\infty)$, depending only on $M$ and $(\alpha_k)_{1\le k\le M}$, such that
\begin{equation}\label{eq:per.ruin}
C^{-1}|f|^{-D-1}\le \P(f\in\mathcal C_{\rm CP}(\rho))\le C|f|^{-D-1}\quad\text{for every }f\in E.
\end{equation}
By Proposition~\ref{prop.res}, $N_e$ agrees with the number of backward crossings of $e$ made by $\X^{(e^+)}$, the extension on $\mathcal P_{e^+}$\textcolor{purple}{,} before its first return to $\rho$. Since the forward and backward crossing counts of $e$ are equal, by the definition of the Markov chain $Z^{(n)}$ we have $N_e=Z_{n-1}^{(n)}$ for this extension. By Proposition~\ref{prop.moment-drift},
$$
\E[N_e^{2(D+1)}]\le Cn^{D+1}.
$$
Since $\{N_e\ge1\}=\{e\in\mathcal C_{\rm CP}(\rho)\}$ almost surely, using \eqref{eq:momentZ}, \eqref{eq:per.ruin} and Proposition~\ref{prop.moment-drift}, we obtain
\begin{equation}\label{eq:per.moment}
\begin{aligned}
\E\left[\left(1+\frac{N_e}{n}\right)^{2D+2}\mid N_e\ge1\right]
&\le 2^{2D+1}\left( 1+\frac{\E[N_e^{2D+2}]}{n^{2D+2}\P(N_e\ge1)}
\right)\le C.
\end{aligned}
\end{equation}
For every $j\ge1$ and $i\in\{1,2\}$, we have
$$
1\wedge\frac{j}{m_i}\le\frac{n+j}{n+m_i-1}=\frac{n}{|e_i|}\left(1+\frac{j}{n}\right).
$$
Using this inequality, \eqref{eq:branch-mixture}, Lemma~\ref{lem:branch_corr} and \eqref{eq:per.moment}, we obtain
\begin{equation}\label{eq:per.joint}
\begin{aligned}
\P\bigl(e_1,e_2\in\mathcal C_{\rm CP}(\rho)\mid e\in\mathcal C_{\rm CP}(\rho)\bigr)&\le C\E\left[\prod_{i=1}^2\left(1\wedge\left(\frac{N_e}{m_i}\right)^{D+1}\right)\mid N_e\ge1\right]\\
&\le C\prod_{i=1}^2\left(\frac{n}{|e_i|}\right)^{D+1}\E\left[ \left(1+\frac{N_e}{n}\right)^{2D+2} \mid N_e\ge1\right]\\
&\le C\prod_{i=1}^2\left(\frac{n}{|e_i|}\right)^{D+1}.
\end{aligned}
\end{equation}
On the other hand, by Proposition~\ref{prop.res}, we have
$\{e_i\in\mathcal C_{\rm CP}(\rho)\} \subseteq \{e\in\mathcal C_{\rm CP}(\rho)\}.
$
Using this inclusion and \eqref{eq:per.ruin}, we obtain
$$
\P\bigl(e_i\in\mathcal C_{\rm CP}(\rho)\mid e\in\mathcal C_{\rm CP}(\rho)\bigr)=\frac{\P(e_i\in\mathcal C_{\rm CP}(\rho))}{\P(e\in\mathcal C_{\rm CP}(\rho))}
\ge C^{-2}\left(\frac{n}{|e_i|}\right)^{D+1}.
$$
Combining this inequality with \eqref{eq:per.joint}, we obtain
$$
\begin{aligned}
\P\bigl(e_1,e_2\in\mathcal C_{\rm CP}(\rho)
\mid e\in\mathcal C_{\rm CP}(\rho)\bigr)\le C\prod_{i=1}^2 \P\bigl(e_i\in\mathcal C_{\rm CP}(\rho) \mid e\in\mathcal C_{\rm CP}(\rho)\bigr), \end{aligned}
$$
where $C$ is independent of $e_1,e_2$. Thus the ruin percolation is quasi-independent.    
\end{proof}
\subsection{Proof of Theorem \ref{main.thm}}\label{sec:proof-thm}

By the same argument as Proposition 11.1 in \cite{CHK2019}, we obtain the following result. We include a proof here for the sake of completeness.

\begin{lemma}[A $0-1$ law for recurrence and transience]\label{lem:zero.one}
Let $\X$ be a MERW. The event that every vertex is visited infinitely often happens with probability $0$ or $1$. 
\end{lemma}
\begin{proof}
Regardless of the current states of the cookie stacks, the probability of a step between two fixed adjacent vertices is bounded below by a positive constant. Indeed, since $\lambda_u^{(j)}=1$ for $j\ge M+1$, for $u\sim w$ we have
$$
\inf_{j\ge1}p_{\boldsymbol\lambda}(u,w,j)
=\min_{1\le j\le M+1}p_{\boldsymbol\lambda}(u,w,j)>0.
$$
Recall that $\tau_u^{(j)}$ denotes the $j$-th visit time of the process $\X$ to $u$. Section~\ref{sec:rubin}. For integers $q\ge0$ and $r\ge1$, by \eqref{P.quenched} we have
$$
\begin{aligned}
&\P\left(\tau_u^{(q+r)}<\infty,\ X_{\tau_u^{(j)}+1}\ne w\text{ for every }q<j\le q+r\right)\le\left(1-\min_{1\le j\le M+1}p_{\boldsymbol\lambda}(u,w,j)\right)^r.
\end{aligned}
$$
Taking the limit as $r\to\infty$, we obtain
$$\begin{aligned}
&\P\left(\lim_{n\rightarrow\infty}\mathsf L_n(u)=\infty,\ \sum_{k=0}^{\infty}\mathds1_{\{X_k=u,\,X_{k+1}=w\}}<\infty\right)\\
& =\P\Big(\bigcup_{q=0}^{\infty} \left\{ \lim_{n\rightarrow\infty}\mathsf L_n(u)=\infty,\ X_{\tau_u^{(j)}+1}\ne w \text{ for every }j>q \right\}\Big)=0
\end{aligned}$$
Thus, almost surely, if $u$ is visited by $\X$ infinitely often, then the step from $u$ to $w$ occurs infinitely often. 
Hence, almost surely, $\X$ visits one vertex infinitely often if and only if $\X$ visits every vertex infinitely often. Consequently, 
\begin{align}\label{eq:recurrence}\P\Big(\{\text{$\X$ visits every vertex infinitely often}\}\cup \{\text{$\X$ visits every vertex finitely often}\}\Big)=1.
\end{align}

For $v\ne\rho$, let $\mathcal T_v$ be the subtree consisting of $v^{-1}$, $v$, and all descendants of $v$. Let $\X^{\mathcal T_v}$ be the extension process to this subtree, constructed similarly as for a path in Section~\ref{sec:ext} with the same family of random variables $\boldsymbol{\xi}$. The same restriction principle as in Proposition~\ref{prop.res} applies to the process $\X^{\mathcal T_v}$.  In particular, when $\mathcal T_v$ is visited by $\X$ infinitely often, we have
\begin{equation}\label{eq:subtree.visits}
\sum_{k\ge0}\mathds1_{\{X_k^{\mathcal T_v}=v^{-1}\}} = 1+\sum_{k\ge0}\mathds1_{\{X_k=v,\,X_{k+1}=v^{-1}\}}.
\end{equation}
Define the event $$B:= \left\{ \operatorname{card}\Big\{ v\in V\setminus\{\rho\}: \sum_{k\ge0}\mathds1_{\{X_k^{\mathcal T_v}=v^{-1}\}}<\infty \Big\} =\infty
\right\}.$$
If every vertex is visited by $\X$ infinitely often, the right-hand side of \eqref{eq:subtree.visits} is infinite for every $v\ne\rho$, and therefore $B$ cannot occur in this case. Consider now the case when every vertex is visited by $\X$ finitely often. Fix $n\ge 0$. Since the vertex set $\{u\in V:|u|<n\}$ is visited only finitely often by $\X$,  there exists a vertex $v$ with $|v|=n$ such that
$X_k\in\mathcal T_v\setminus\{v^{-1}\}$ for all sufficiently large $k$. By \eqref{eq:subtree.visits}, we thus have
$$
\operatorname{card}\big\{k\ge0:X_k^{\mathcal T_v}=v^{-1}\big\}\le 1+\sum_{k\ge0}\mathds1_{\{X_k=v^{-1}\}}<\infty.
$$ 
Hence $B$ occurs in this case. 
Consequently, almost surely,
\begin{equation}\label{eq:B.transience}
B=\{\X\text{ visits every vertex finitely often}\}.
\end{equation}

We finally prove that $B$ belongs to the tail sigma-algebra of $\boldsymbol\xi$. For every $n\ge0$, removing vertices with $|v|\le n$ from its definition leaves $B$ unchanged. Hence
$$
\begin{aligned}
B &=\left\{\operatorname{card}\Big\{v\in V:|v|>n,\ \sum_{k\ge0}\mathds1_{\{X_k^{\mathcal T_v}=v^{-1}\}}<\infty\Big\}=\infty\right\}\\
&\in \sigma\bigl(\xi(u,w,j):|u|>n,\ w\sim u,\ j\ge1\bigr).
\end{aligned}
$$
Therefore $B$ belongs to the tail sigma-algebra of the independent random variables $\boldsymbol\xi$. By Kolmogorov's $0-1$ law, we have $\P(B)\in\{0,1\}$. Using \eqref{eq:recurrence} and \eqref{eq:B.transience}, we conclude that
$$
\P(\X\text{ visits every vertex infinitely often})
=1-\P(B)\in\{0,1\}.
$$

\end{proof}

We are now ready for the proof of the main result.

\begin{proof}[Proof of Theorem \ref{main.thm}]
Let $\mathbf Q$ denote the law of the ruin percolation. By its definition and Proposition~\ref{prop.res}, we have
$$
\mathbf Q(\rho\leftrightarrow e^+)
=\P(e\in\mathcal C_{\rm CP}(\rho))
\quad\text{for every }e\in E.
$$
By Proposition \ref{thm:ruin_prob}, there exists $C\in[1,\infty)$ such that
\begin{align}
C^{-1}|e|^{-D-1}\label{ruininq}
\le \mathbf Q(\rho\leftrightarrow e^+)
\le C|e|^{-D-1}
\quad\text{for every }e\in E.
\end{align}
By Proposition~\ref{prop:per}, the ruin percolation is quasi-independent, so the criteria in Proposition~\ref{lyons1989} apply.

\textbf{Case 1: when ${\rm br}_r(\mathcal T)<D+1$}. We choose $\gamma$ such that
${\rm br}_r(\mathcal T)<\gamma<D+1.$
By \eqref{brr2}, we have
$$
\inf_{\pi\in\Pi}\sum_{e\in\pi}|e|^{-\gamma}=0.
$$
Using \eqref{eq:per.ruin} and $\gamma<D+1$, we obtain
$$
\inf_{\pi\in\Pi}\sum_{e\in\pi}\mathbf Q(\rho\leftrightarrow e^+)\le C\inf_{\pi\in\Pi}\sum_{e\in\pi}|e|^{-D-1}\le C\inf_{\pi\in\Pi}\sum_{e\in\pi}|e|^{-\gamma}=0.
$$
By Proposition~\ref{lyons1989}(i), we conclude that $\mathbf Q(\rho\leftrightarrow\infty)=0$. By Lemma~\ref{per2.lem} and Lemma \ref{lem:zero.one}, the process $\X$ is almost surely recurrent.

\textbf{Case 2: when ${\rm br}_r(\mathcal T)>D+1$}. We choose $\gamma$ such that
$D+1<\gamma<{\rm br}_r(\mathcal T).$
By \eqref{brr2}, we have
$$
\inf_{\pi\in\Pi}\sum_{e\in\pi}|e|^{-\gamma}>0.
$$
By the max-flow min-cut theorem, there exists a non-zero flow $\theta$ from $\rho$ to infinity satisfying
$$
0\le\theta(e)\le|e|^{-\gamma}
\quad\text{for every }e\in E.
$$

Recall the adapted conductances defined in \eqref{adt.c}. For $|e|>1$, we have
\begin{equation}\label{eq:main.conductance}
\begin{aligned}
\frac1{c(e)}
&=
\frac{\mathbf Q(e\text{ is closed}\mid\rho\leftrightarrow e^-)}
{\mathbf Q(\rho\leftrightarrow e^+)}=
\frac1{\mathbf Q(\rho\leftrightarrow e^+)}
-\frac1{\mathbf Q(\rho\leftrightarrow e^-)}.
\end{aligned}
\end{equation}
Fix $v\ne\rho$, and let $e_1,\ldots,e_N$ be the successive edges of $\mathcal P_v$, where $N=|v|$ and $|e_n|=n$. By \eqref{eq:main.conductance} and using summation by parts, we obtain
\begin{equation}\label{eq:main.potential}
\begin{aligned}
\sum_{e\in\mathcal P_v}\frac{\theta(e)}{c(e)}
&\le\sum_{n=1}^N\frac{n^{-\gamma}}{c(e_n)}\le \frac{N^{-\gamma}}{\mathbf Q(\rho\leftrightarrow v)} +\sum_{n=1}^{N-1}\frac{n^{-\gamma}-(n+1)^{-\gamma}}{\mathbf Q(\rho\leftrightarrow e_n^+)}\\
&\le CN^{D+1-\gamma}+C\sum_{n=1}^{N-1}n^{D-\gamma}\le C.
\end{aligned}
\end{equation}
For the last two inequalities in \eqref{eq:main.potential}, we used \eqref{eq:per.ruin},
$n^{-\gamma}-(n+1)^{-\gamma}\le\gamma n^{-\gamma-1},$ and $\gamma>D+1$. In particular, the constant in \eqref{eq:main.potential} is independent of $v$.

Let $\mathfrak m_\theta$ be the measure induced by the flow $\theta$ on $\partial\mathcal T$, characterised by
$$
\mathfrak m_\theta\bigl(\{\xi\in\partial\mathcal T:e\in\xi\}\bigr)=\theta(e)
\quad\text{for every }e\in E.
$$
For $\xi\in\partial\mathcal T$, set
$$
V_\theta(\xi):=\sum_{e\in\xi}\frac{\theta(e)}{c(e)}.
$$
By \eqref{eq:main.potential}, we have $V_\theta(\xi)\le C$ for every $\xi\in\partial\mathcal T$. Hence,
$$
\begin{aligned}
\sum_{e\in E}\frac{\theta(e)^2}{c(e)} &=\int_{\partial\mathcal T}V_\theta(\xi)\,\mathrm d\mathfrak m_\theta(\xi)\le C\mathfrak m_\theta(\partial\mathcal T) =C\sum_{|e|=1}\theta(e)<\infty.
\end{aligned}
$$
By Proposition~\ref{lyons1989}(ii), we thus have $\mathbf Q(\rho\leftrightarrow\infty)>0$. Combining this with Lemma~\ref{per2.lem} and Lemma \ref{lem:zero.one}, we conclude that the process $\X$ is almost surely transient.
\end{proof}

\section{Appendix}
To convert the asymptotics of some of generating functions we used to asymptotics of their coefficients, we will use Karamata's Tauberin theorem, see corollary 1.7.3 of \cite{BGT1987}. 
\begin{theorem}[Karamata Tauberian Theorem \cite{BGT1987}]\label{thm:karamata}
Let $(a_n)_{n\ge0}$ be a sequence of non-negative real numbers which is ultimately non-increasing. If, for some $c>0$ and $\rho>0$,
$$
\sum_{m=n}^{\infty}a_m\sim c n^{-\rho} \quad\text{as }n\rightarrow\infty,
$$
then
$$
a_n\sim c\rho n^{-\rho-1}
\quad\text{as }n\rightarrow\infty.
$$
\end{theorem}

\begin{lemma}\label{lem:lypint}
Assume that $D>0$. For every $a\ge1$ and $q\in(0,D)$, there exists $R<\infty$ such that, for every integer $y>R$,
$$\int_0^1(-\log u)^{q-1}u^{a-1}(2-u)^{-(y+1)}\prod_{j=1}^M(1-\beta_j(1-u))\rmd u\le\Gamma(q)(y+a)^{-q}.$$
\end{lemma}

\begin{proof}
Fix $a\ge1$ and $q\in(0,D)$. Let $K_y$ be a negative binomial random variable with
$$
\P(K_y=k)=\binom{y+k}{k}2^{-(y+1+k)}
\quad \text{for }k\ge0.
$$
Let $J$ be the sum of independent $\operatorname{Bernoulli}(\beta_j)$ random variables, $1\le j\le M$, independently of $K_y$. By the binomial expansion, we note that
$$
\E[u^{K_y+J}] =(2-u)^{-(y+1)} \prod_{j=1}^M(1-\beta_j(1-u)).
$$
Hence
\begin{align*}
&\int_0^1(-\log u)^{q-1}u^{a-1}(2-u)^{-(y+1)}
\prod_{j=1}^M(1-\beta_j(1-u))\rmd u\\
&\quad= \E\left[\int_0^1(-\log u)^{q-1}u^{K_y+J+a-1}\rmd u\right]=\Gamma(q)\E[(K_y+J+a)^{-q}].
\end{align*}
It therefore suffices to prove that
$$
\E[(K_y+J+a)^{-q}]\le(y+a)^{-q}
$$
for all sufficiently large $y$.

Using the Taylor expansion and taking expectations, we have
\begin{align*}
\E[(K_y+J+a)^{-q}]
&=(y+a)^{-q}
-q(y+a)^{-q-1}\E[K_y+J-y]+\frac{q(q+1)}2(y+a)^{-q-2}\E[(K_y+J-y)^2]\\
&-\frac{q(q+1)(q+2)}2\E\Big[ (K_y+J-y)^3
\int_0^1(1-u)^2
\bigl(y+a+u(K_y+J-y)\bigr)^{-q-3}\rmd u\Big].
\end{align*}
Since $\E[K_y]=y+1$, $\operatorname{Var}(K_y)=2(y+1)$ and $\E[J]=D$, we have
$$
\E[K_y+J-y]=D+1,
\quad
\E[(K_y+J-y)^2]=2y+O(1).
$$
Moreover, $K_y$ is a sum of $y+1$ independent geometric random variables of mean $1$. We have 
$$
\E[(K_y-y-1)^4]=O(y^2).
$$
Using $0\le J\le M$ and H\"older's inequality, we obtain
$$
\E[|K_y+J-y|^3]=O(y^{3/2}).
$$
To bound the remainder term, we first note that on $\{K_y\ge y/2\}$, every point between $y+a$ and $K_y+J+a$ is at least $y/2+a$, and thus
$$\int_0^1(1-u)^2
\bigl(y+a+u(K_y+J-y)\bigr)^{-q-3}\rmd u\le Cy^{-q-3}$$
On $\{K_y<y/2\}$, this remainder is bounded by a constant, since $K_y+J+a\ge a\ge1$ and $|K_y+J-y|\le y+M$. Using the fact that
$$
\P(K_y<y/2)
\le 2^{y/2}\E[2^{-K_y}]
=\frac23\left(\frac{2\sqrt2}{3}\right)^y,
$$
we  thus obtain $$\E\Big[ (K_y+J-y)^3
\int_0^1(1-u)^2
\bigl(y+a+u(K_y+J-y)\bigr)^{-q-3}\rmd u\Big]=O(y^{-q-3/2}).$$
Therefore, 
\begin{align*}
\E[(K_y+J+a)^{-q}]
&=(y+a)^{-q}
-q(D+1)(y+a)^{-q-1}+\frac{q(q+1)}2(y+a)^{-q-2}\bigl(2y+O(1)\bigr)
+O(y^{-q-3/2})\\
&=(y+a)^{-q}
\left(1+\frac{q(q-D)}{y}+O(y^{-3/2})\right).
\end{align*}
Since $q(q-D)<0$, we conclude that
$\E[(K_y+J+a)^{-q}]\le(y+a)^{-q}$
for all sufficiently large $y$. This completes the proof of the lemma.

\end{proof}

\begin{lemma}\label{lem:Gamma}
For $a,b\in\C$ with $|\operatorname{Re}(a)|\le\alpha<\infty$, $|\operatorname{Re}(b)|\le\alpha$, $|\operatorname{Im}(a)|\le1$, $|\operatorname{Im}(b)|\le1$ and $\operatorname{Re}(a-b)\le0$, we have
$$
\frac{\Gamma(n+a)}{\Gamma(n+b)}
=n^{a-b}\left(1+O(1/n)\right)
\quad\text{as }n\rightarrow\infty,
$$
where the error term is uniform for $a,b$.
\end{lemma}
\begin{proof}
We use the following version of Stirling's formula:
$$
\log\Gamma(z)
=\left(z-\frac12\right)\log z-z+\frac12\log(2\pi)+O(1/z).
$$
Here $\log z$ and $\log\Gamma(z)$ are chosen analytic on $\operatorname{Re}(z)>0$ and real for $z>0$.
All error terms below are uniform in $a,b$. We have
\begin{align*}
\log\Gamma(n+a)
&=\left(n+a-\frac12\right)\log(n+a)-(n+a)
+\frac12\log(2\pi)+O(1/n)\\
&=\left(n+a-\frac12\right)
\left(\log n+\frac{a}{n}+O(n^{-2})\right)
-(n+a)+\frac12\log(2\pi)+O(1/n)\\
&=\left(n+a-\frac12\right)\log n-n
+\frac12\log(2\pi)+O(1/n),
\end{align*}
and similarly
$$
\log\Gamma(n+b)
=\left(n+b-\frac12\right)\log n-n
+\frac12\log(2\pi)+O(1/n).
$$
Subtracting, we obtain
$$
\log\Gamma(n+a)-\log\Gamma(n+b)
=(a-b)\log n+O(1/n).
$$
Exponentiating both sides, we conclude that
$$
\frac{\Gamma(n+a)}{\Gamma(n+b)}
=n^{a-b}e^{O(1/n)}
=n^{a-b}\left(1+O(1/n)\right).
$$
\end{proof}

\bibliography{refs}

@article{Kozma2003,
  title={Excited random walk in three dimensions has positive speed},
  author={Kozma, Gady},
  journal={arXiv preprint math/0310305},
  year={2003}
}

@article{Kozma2005,
  title={Excited random walk in two dimensions has linear speed},
  author={Kozma, Gady},
  journal={arXiv preprint math/0512535},
  year={2005}
}

@article {Pham2015,
    AUTHOR = {Pham, Cong-Dan},
     TITLE = {Monotonicity and regularity of the speed for excited random
              walks in higher dimensions},
   JOURNAL = {Electron. J. Probab.},
  FJOURNAL = {Electronic Journal of Probability},
    VOLUME = {20},
      YEAR = {2015},
     PAGES = {no. 72, 25},
      ISSN = {1083-6489},
   MRCLASS = {60K37},
  MRNUMBER = {3361260},
MRREVIEWER = {Dimitri\ Petritis},
       DOI = {10.1214/EJP.v20-3522},
       URL = {https://doi.org/10.1214/EJP.v20-3522},
}

@article{BNSP2026,
    title = {Limit shapes for centrally excited random walks},
    author = {Bou-Rabee, Ahmed and Nguyen, Tuan-Minh and Peres, Yuval and Singh, Arvind},
    journal={arXiv preprint arXiv:2610.10162v2},
    year={2026}
}

@article{N2026b,
  title={Shape theorem for excited random walks on the integer lattice},
  author={Nguyen, Tuan-Minh},
  journal={arXiv preprint arXiv:2610.07724},
  year={2026}
}

@article{N2026,
  title={``{T}rue" self-avoiding walks on general trees},
  author={Nguyen, Tuan-Minh},
  journal={arXiv preprint arXiv:2604.24389},
  year={2026}
}

@article{Campos2021, 
   AUTHOR = {Campos, Daniel and Crist\'in, Javier and M\'endez, Vicen\c c},
     TITLE = {Optimal escape-and-feeding dynamics of random walkers:
              rethinking the convenience of ballistic strategies},
   JOURNAL = {Phys. Rev. E},
  FJOURNAL = {Physical Review E},
    VOLUME = {103},
      YEAR = {2021},
    NUMBER = {5},
     PAGES = {Paper No. 052109, 8},
      ISSN = {2470-0045,2470-0053},
   MRCLASS = {60G50},
  MRNUMBER = {4277648},
       DOI = {10.1103/physreve.103.052109},
       URL = {https://doi.org/10.1103/physreve.103.052109},
}

@article{BenichouRedner2014, 
author = {B{\'e}nichou, Olivier and Redner, Sidney}, 
title = {Depletion-Controlled Starvation of a Diffusing Forager}, 
journal = {Physical Review Letters}, 
volume = {113}, 
pages = {238101}, 
year = {2014}, 
doi = {10.1103/PhysRevLett.113.238101}, 
url = {https://doi.org/10.1103/PhysRevLett.113.238101} }

@article{Bhat2017, 
    AUTHOR = {Bhat, U. and Redner, S. and B\'enichou, O.},
     TITLE = {Starvation dynamics of a greedy forager},
   JOURNAL = {J. Stat. Mech. Theory Exp.},
  FJOURNAL = {Journal of Statistical Mechanics: Theory and Experiment},
      YEAR = {2017},
    NUMBER = {7},
    volume = {2017},
     PAGES = {073213},
      ISSN = {1742-5468},
   MRCLASS = {82C41},
  MRNUMBER = {3683815},
       DOI = {10.1088/1742-5468/aa7dfc},
       URL = {https://doi.org/10.1088/1742-5468/aa7dfc},
}

@article{Bracis2015, 
author = {Bracis, Chloe and Gurarie, Eliezer and Van Moorter, Bram and Goodwin, R. Andrew}, 
title = {Memory Effects on Movement Behavior in Animal Foraging}, 
journal = {PLOS ONE}, 
volume = {10}, 
number = {8}, 
pages = {e0136057}, 
year = {2015}, 
doi = {10.1371/journal.pone.0136057}, 
url = {https://doi.org/10.1371/journal.pone.0136057} }

@article{VanMoorter2009, 
author = {Van Moorter, Bram and Visscher, Darcy and Benhamou, Simon and B{\"o}rger, Luca and Boyce, Mark S. and Gaillard, Jean-Michel}, 
title = {Memory keeps you at home: a mechanistic model for home range emergence}, 
journal = {Oikos}, 
volume = {118}, 
pages = {641--652}, 
year = {2009}, 
doi = {10.1111/j.1600-0706.2008.17003.x}, 
url = {https://doi.org/10.1111/j.1600-0706.2008.17003.x} }

@article{Fronhofer2013, 
author = {Fronhofer, Emanuel A. and Hovestadt, Thomas and Poethke, Hans-Joachim},
title = {From random walks to informed movement},
journal = {Oikos},
volume = {122},
number = {6},
pages = {857-866},
doi = {10.1111/j.1600-0706.2012.21021.x},
url = {https://doi.org/10.1111/j.1600-0706.2012.21021.x},
year = {2013}
}

@article{GetzSaltz2008, author = {Getz, Wayne M. and Saltz, David}, title = {A framework for generating and analyzing movement paths on ecological landscapes}, journal = {Proceedings of the National Academy of Sciences}, volume = {105}, number = {49}, pages = {19066--19071}, year = {2008}, doi = {10.1073/pnas.0801732105}, url = {https://doi.org/10.1073/pnas.0801732105} }

@article{BartumeusLevin2008, author = {Bartumeus, Frederic and Levin, Simon A.}, title = {Fractal reorientation clocks: Linking animal behavior to statistical patterns of search}, journal = {Proceedings of the National Academy of Sciences}, volume = {105}, number = {49}, pages = {19072--19077}, year = {2008}, doi = {10.1073/pnas.0801926105}, url = {https://doi.org/10.1073/pnas.0801926105} }

@article {BW2003,
    AUTHOR = {Benjamini, Itai and Wilson, David B.},
     TITLE = {Excited random walk},
   JOURNAL = {Electron. Comm. Probab.},
  FJOURNAL = {Electronic Communications in Probability},
    VOLUME = {8},
      YEAR = {2003},
     PAGES = {86--92},
      ISSN = {1083-589X},
   MRCLASS = {60G50 (60K37)},
  MRNUMBER = {1987097},
MRREVIEWER = {Michael Voit},
       DOI = {10.1214/ECP.v8-1072},
       URL = {https://doi.org/10.1214/ECP.v8-1072},
}

@article {Z2005,
    AUTHOR = {Zerner, Martin P. W.},
     TITLE = {Multi-excited random walks on integers},
   JOURNAL = {Probab. Theory Related Fields},
  FJOURNAL = {Probability Theory and Related Fields},
    VOLUME = {133},
      YEAR = {2005},
    NUMBER = {1},
     PAGES = {98--122},
      ISSN = {0178-8051},
   MRCLASS = {60K35 (60J10 60K37)},
  MRNUMBER = {2197139},
MRREVIEWER = {Serguei Popov},
       DOI = {10.1007/s00440-004-0417-0},
       URL = {https://doi.org/10.1007/s00440-004-0417-0},
}

@article {CKS2020,
    AUTHOR = {Collevecchio, Andrea and Kious, Daniel and Sidoravicius,
              Vladas},
     TITLE = {The branching-ruin number and the critical parameter of
              once-reinforced random walk on trees},
   JOURNAL = {Comm. Pure Appl. Math.},
  FJOURNAL = {Communications on Pure and Applied Mathematics},
    VOLUME = {73},
      YEAR = {2020},
    NUMBER = {1},
     PAGES = {210--236},
      ISSN = {0010-3640},
   MRCLASS = {60K35 (05C81)},
  MRNUMBER = {4033893},
       DOI = {10.1002/cpa.21860},
       URL = {https://doi.org/10.1002/cpa.21860},
}

@article {DK2012,
    AUTHOR = {Dolgopyat, Dmitry and Kosygina, Elena},
     TITLE = {Scaling limits of recurrent excited random walks on integers},
   JOURNAL = {Electron. Commun. Probab.},
  FJOURNAL = {Electronic Communications in Probability},
    VOLUME = {17},
      YEAR = {2012},
     PAGES = {no. 35, 14},
   MRCLASS = {60K37 (60F17 60G50)},
  MRNUMBER = {2965748},
MRREVIEWER = {Bruno Schapira},
       DOI = {10.1214/ECP.v17-2213},
       URL = {https://doi.org/10.1214/ECP.v17-2213},
}

@article {KM2011,
    AUTHOR = {Kosygina, Elena and Mountford, Thomas},
     TITLE = {Limit laws of transient excited random walks on integers},
   JOURNAL = {Ann. Inst. Henri Poincar\'{e} Probab. Stat.},
  FJOURNAL = {Annales de l'Institut Henri Poincar\'{e} Probabilit\'{e}s et
              Statistiques},
    VOLUME = {47},
      YEAR = {2011},
    NUMBER = {2},
     PAGES = {575--600},
      ISSN = {0246-0203},
   MRCLASS = {60K37 (60F05 60J60 60J80)},
  MRNUMBER = {2814424},
MRREVIEWER = {Jean B\'{e}rard},
       DOI = {10.1214/10-AIHP376},
       URL = {https://doi.org/10.1214/10-AIHP376},
}

@article {P2012,
    AUTHOR = {Peterson, Jonathon},
     TITLE = {Large deviations and slowdown asymptotics for one-dimensional
              excited random walks},
   JOURNAL = {Electron. J. Probab.},
  FJOURNAL = {Electronic Journal of Probability},
    VOLUME = {17},
      YEAR = {2012},
     PAGES = {no. 48, 24},
   MRCLASS = {60F10 (60K37)},
  MRNUMBER = {2946155},
MRREVIEWER = {Ofer Zeitouni},
       DOI = {10.1214/EJP.v17-1726},
       URL = {https://doi.org/10.1214/EJP.v17-1726},
}

@article {DP2017,
AUTHOR = {Davis, B. and Peterson, J.},
TITLE = {Excited random walks with non-nearest neighbor steps},
JOURNAL = {J. Theoret. Probab.},
FJOURNAL = {Journal of Theoretical Probability},
VOLUME = {30},
YEAR = {2017},
NUMBER = {4},
PAGES = {1255--1284},
ISSN = {0894-9840},
MRCLASS = {60K35 (60G50 60K37)},
MRNUMBER = {3736173},
MRREVIEWER = {Bruno Schapira},
DOI = {10.1007/s10959-016-0697-1},
URL = {https://doi.org/10.1007/s10959-016-0697-1},
}

@article {KP2017,
    AUTHOR = {Kosygina, E. and Peterson, J.},
     TITLE = {Excited random walks with {M}arkovian cookie stacks},
   JOURNAL = {Ann. Inst. Henri Poincar\'{e} Probab. Stat.},
  FJOURNAL = {Annales de l'Institut Henri Poincar\'{e} Probabilit\'{e}s et
              Statistiques},
    VOLUME = {53},
      YEAR = {2017},
    NUMBER = {3},
     PAGES = {1458--1497},
      ISSN = {0246-0203},
   MRCLASS = {60K37 (60F05 60J10 60K35)},
  MRNUMBER = {3689974},
MRREVIEWER = {Bruno Schapira},
       DOI = {10.1214/16-AIHP761},
       URL = {https://doi.org/10.1214/16-AIHP761},
}

@article {KMP2022,
    AUTHOR = {Kosygina, E. and Mountford, T. and Peterson, J.},
     TITLE = {Convergence of random walks with {M}arkovian cookie stacks to
              {B}rownian motion perturbed at extrema},
   JOURNAL = {Probab. Theory Related Fields},
  FJOURNAL = {Probability Theory and Related Fields},
    VOLUME = {182},
      YEAR = {2022},
    NUMBER = {1-2},
     PAGES = {189--275},
      ISSN = {0178-8051},
   MRCLASS = {60K35 (60F17 60J55)},
  MRNUMBER = {4367948},
       DOI = {10.1007/s00440-021-01055-3},
       URL = {https://doi.org/10.1007/s00440-021-01055-3},
}

@article {Hofstad2010,
    AUTHOR = {van der Hofstad, R. and Holmes, M.},
     TITLE = {Monotonicity for excited random walk in high dimensions},
   JOURNAL = {Probab. Theory Related Fields},
  FJOURNAL = {Probability Theory and Related Fields},
    VOLUME = {147},
      YEAR = {2010},
    NUMBER = {1-2},
     PAGES = {333--348},
      ISSN = {0178-8051},
   MRCLASS = {60K35 (60G50 60K37)},
  MRNUMBER = {2594356},
MRREVIEWER = {Daniel Remenik},
       DOI = {10.1007/s00440-009-0215-9},
       URL = {https://doi.org/10.1007/s00440-009-0215-9},
}

@article {CHK2019,
    AUTHOR = {Collevecchio, Andrea and Huynh, Cong Bang and Kious, Daniel},
     TITLE = {The branching-ruin number as critical parameter of random
              processes on trees},
   JOURNAL = {Electron. J. Probab.},
  FJOURNAL = {Electronic Journal of Probability},
    VOLUME = {24},
      YEAR = {2019},
     PAGES = {Paper No. 121, 29},
      ISSN = {1083-6489},
   MRCLASS = {60K35 (60K37 82D30)},
  MRNUMBER = {4029424},
MRREVIEWER = {Daniel\ Boivin},
       DOI = {10.1214/19-ejp383},
       URL = {https://doi.org/10.1214/19-ejp383},
}

@article {Lyons1989,
    AUTHOR = {Lyons, Russell},
     TITLE = {The {I}sing model and percolation on trees and tree-like
              graphs},
   JOURNAL = {Comm. Math. Phys.},
  FJOURNAL = {Communications in Mathematical Physics},
    VOLUME = {125},
      YEAR = {1989},
    NUMBER = {2},
     PAGES = {337--353},
      ISSN = {0010-3616,1432-0916},
   MRCLASS = {82A43 (60K35 82A05 82A68)},
  MRNUMBER = {1016874},
MRREVIEWER = {R.\ T.\ Smythe},
       URL = {http://projecteuclid.org/euclid.cmp/1104179469},
}

@article {ABK2008,
    AUTHOR = {Amir, Gideon and Benjamini, Itai and Kozma, Gady},
     TITLE = {Excited random walk against a wall},
   JOURNAL = {Probab. Theory Related Fields},
  FJOURNAL = {Probability Theory and Related Fields},
    VOLUME = {140},
      YEAR = {2008},
    NUMBER = {1-2},
     PAGES = {83--102},
      ISSN = {0178-8051,1432-2064},
   MRCLASS = {60G50},
  MRNUMBER = {2357671},
MRREVIEWER = {Serguei\ Popov},
       DOI = {10.1007/s00440-007-0058-1},
       URL = {https://doi.org/10.1007/s00440-007-0058-1},
}

@article {AHR2023,
    AUTHOR = {Angel, Omer and Holmes, Mark and Ramirez, Alejandro},
     TITLE = {Balanced excited random walk in two dimensions},
   JOURNAL = {Ann. Probab.},
  FJOURNAL = {The Annals of Probability},
    VOLUME = {51},
      YEAR = {2023},
    NUMBER = {4},
     PAGES = {1421--1448},
      ISSN = {0091-1798,2168-894X},
   MRCLASS = {60K35 (60G42)},
  MRNUMBER = {4597323},
MRREVIEWER = {Andrew\ R.\ Wade},
       DOI = {10.1214/23-aop1622},
       URL = {https://doi.org/10.1214/23-aop1622},
}

@article {Z2006,
    AUTHOR = {Zerner, Martin P. W.},
     TITLE = {Recurrence and transience of excited random walks on {$\mathbb{Z}^d$} and strips},
   JOURNAL = {Electron. Comm. Probab.},
  FJOURNAL = {Electronic Communications in Probability},
    VOLUME = {11},
      YEAR = {2006},
     PAGES = {118--128},
      ISSN = {1083-589X},
   MRCLASS = {60K35 (60G50 60K37)},
  MRNUMBER = {2231739},
MRREVIEWER = {Marina\ Vachkovskaia},
       DOI = {10.1214/ECP.v11-1200},
       URL = {https://doi.org/10.1214/ECP.v11-1200},
}

@article {BS2009,
    AUTHOR = {Basdevant, Anne-Laure and Singh, Arvind},
     TITLE = {Recurrence and transience of a multi-excited random walk on a
              regular tree},
   JOURNAL = {Electron. J. Probab.},
  FJOURNAL = {Electronic Journal of Probability},
    VOLUME = {14},
      YEAR = {2009},
     PAGES = {no. 55, 1628--1669},
      ISSN = {1083-6489},
   MRCLASS = {60K35 (60F05 60F15 60F20 60G50 60J80)},
  MRNUMBER = {2525106},
MRREVIEWER = {Da-Quan\ Jiang},
       DOI = {10.1214/EJP.v14-672},
       URL = {https://doi.org/10.1214/EJP.v14-672},
}

@book {LP2016,
    AUTHOR = {Lyons, Russell and Peres, Yuval},
     TITLE = {Probability on trees and networks},
    SERIES = {Cambridge Series in Statistical and Probabilistic Mathematics},
    VOLUME = {42},
 PUBLISHER = {Cambridge University Press, New York},
      YEAR = {2016},
     PAGES = {xv+699},
      ISBN = {978-1-107-16015-6},
   MRCLASS = {60C05 (05C05 05C81 28A80 60J10 60J80 60K35 82B41)},
  MRNUMBER = {3616205},
MRREVIEWER = {Laurent Miclo},
       DOI = {10.1017/9781316672815},
       URL = {https://doi.org/10.1017/9781316672815},
}

@article {V2003,
    AUTHOR = {Volkov, Stanislav},
     TITLE = {Excited random walk on trees},
   JOURNAL = {Electron. J. Probab.},
  FJOURNAL = {Electronic Journal of Probability},
    VOLUME = {8},
      YEAR = {2003},
     PAGES = {no. 23, 15},
      ISSN = {1083-6489},
   MRCLASS = {60J10 (60G50)},
  MRNUMBER = {2041824},
MRREVIEWER = {Da-Quan\ Jiang},
       DOI = {10.1214/EJP.v8-180},
       URL = {https://doi.org/10.1214/EJP.v8-180},
}

@article {BR2007,
    AUTHOR = {B\'erard, Jean and Ram\'irez, Alejandro},
     TITLE = {Central limit theorem for the excited random walk in dimension
              {$D\geq 2$}},
   JOURNAL = {Electron. Comm. Probab.},
  FJOURNAL = {Electronic Communications in Probability},
    VOLUME = {12},
      YEAR = {2007},
     PAGES = {303--314},
      ISSN = {1083-589X},
   MRCLASS = {60K35 (60F05 60G50)},
  MRNUMBER = {2342709},
MRREVIEWER = {Martin\ V.\ Hildebrand},
       DOI = {10.1214/ECP.v12-1317},
       URL = {https://doi.org/10.1214/ECP.v12-1317},
}

@article {Menshikov2012,
    AUTHOR = {Menshikov, M. and Popov, S. and Ram\'{\i}rez, A.
              F. and Vachkovskaia, M.},
     TITLE = {On a general many-dimensional excited random walk},
   JOURNAL = {Ann. Probab.},
  FJOURNAL = {The Annals of Probability},
    VOLUME = {40},
      YEAR = {2012},
    NUMBER = {5},
     PAGES = {2106--2130},
      ISSN = {0091-1798},
   MRCLASS = {60G50 (60K37)},
  MRNUMBER = {3025712},
MRREVIEWER = {David A. Croydon},
       DOI = {10.1214/11-AOP678},
       URL = {https://doi.org/10.1214/11-AOP678},
}

@article {N2022,
    AUTHOR = {Nguyen, Tuan-Minh},
     TITLE = {Speed of excited random walks with long backward steps},
   JOURNAL = {J. Stat. Phys.},
  FJOURNAL = {Journal of Statistical Physics},
    VOLUME = {188},
      YEAR = {2022},
    NUMBER = {1},
     PAGES = {Paper No. 3, 29},
      ISSN = {0022-4715,1572-9613},
   MRCLASS = {60K35 (60J80 60J85 82C41)},
  MRNUMBER = {4419336},
MRREVIEWER = {V.\ Shcherbakov},
       DOI = {10.1007/s10955-022-02926-2},
       URL = {https://doi.org/10.1007/s10955-022-02926-2},
}

@article{CHN2021,
 AUTHOR = {Collevecchio, A. and Hamza, K. and Nguyen, T.-M.},
     TITLE = {Long range one-cookie random walk with positive speed},
   JOURNAL = {Stochastic Process. Appl.},
  FJOURNAL = {Stochastic Processes and their Applications},
    VOLUME = {142},
      YEAR = {2021},
     PAGES = {462--478},
      ISSN = {0304-4149},
   MRCLASS = {60K35 (60F15 60K37)},
  MRNUMBER = {4321331},
       DOI = {10.1016/j.spa.2021.08.003},
       URL = {https://doi.org/10.1016/j.spa.2021.08.003},
}

@article {KZ2013,
    AUTHOR = {Kosygina, Elena and Zerner, Martin P. W.},
     TITLE = {Excited random walks: results, methods, open problems},
   JOURNAL = {Bull. Inst. Math. Acad. Sin. (N.S.)},
  FJOURNAL = {Bulletin of the Institute of Mathematics. Academia Sinica. New
              Series},
    VOLUME = {8},
      YEAR = {2013},
    NUMBER = {1},
     PAGES = {105--157},
      ISSN = {2304-7909},
   MRCLASS = {60K35 (60J80 60K37)},
  MRNUMBER = {3097419},
MRREVIEWER = {Andrew R. Wade},
}

@article {ABO,
    AUTHOR = {Amir, Gideon and Berger, Noam and Orenshtein, Tal},
     TITLE = {Zero-one law for directional transience of one dimensional
              excited random walks},
   JOURNAL = {Ann. Inst. Henri Poincar\'{e} Probab. Stat.},
  FJOURNAL = {Annales de l'Institut Henri Poincar\'{e} Probabilit\'{e}s et
              Statistiques},
    VOLUME = {52},
      YEAR = {2016},
    NUMBER = {1},
     PAGES = {47--57},
      ISSN = {0246-0203},
   MRCLASS = {82B41 (60K35 60K37)},
  MRNUMBER = {3449293},
MRREVIEWER = {Elena A. Zhizhina},
       DOI = {10.1214/14-AIHP615},
       URL = {https://doi.org/10.1214/14-AIHP615},
}

@article {BasS1,
    AUTHOR = {Basdevant, Anne-Laure and Singh, Arvind},
     TITLE = {On the speed of a cookie random walk},
   JOURNAL = {Probab. Theory Related Fields},
  FJOURNAL = {Probability Theory and Related Fields},
    VOLUME = {141},
      YEAR = {2008},
    NUMBER = {3-4},
     PAGES = {625--645},
      ISSN = {0178-8051},
   MRCLASS = {60K35 (60F15 60J80)},
  MRNUMBER = {2391167},
MRREVIEWER = {David A. Croydon},
       DOI = {10.1007/s00440-007-0096-8},
       URL = {https://doi.org/10.1007/s00440-007-0096-8},
}

@article {BasS2,
    AUTHOR = {Basdevant, Anne-Laure and Singh, Arvind},
     TITLE = {Rate of growth of a transient cookie random walk},
   JOURNAL = {Electron. J. Probab.},
  FJOURNAL = {Electronic Journal of Probability},
    VOLUME = {13},
      YEAR = {2008},
     PAGES = {no. 26, 811--851},
   MRCLASS = {60K35 (60F05 60J80)},
  MRNUMBER = {2399297},
MRREVIEWER = {Jean B\'{e}rard},
       DOI = {10.1214/EJP.v13-498},
       URL = {https://doi.org/10.1214/EJP.v13-498},
}

@article{LN2026,
  title={Once-excited random walks on general trees},
  author={Le, Duy-Bao and Nguyen, Tuan-Minh},
  journal={arXiv preprint arXiv:2602.16934},
  year={2026}
}

@book{BGT1987,
  author    = {N. H. Bingham and C. M. Goldie and J. L. Teugels},
  title     = {Regular Variation},
  publisher = {Cambridge University Press},
  year      = {1987}
}
\bibliographystyle{alpha}
\end{document}